\documentclass[letterpaper,11pt,reqno]{amsart} 
\usepackage[portrait,margin=3cm]{geometry} 
\usepackage{mathrsfs,xfrac} 
\usepackage[colorlinks=true,linkcolor=blue,citecolor=blue,urlcolor=blue]{hyperref} 
\usepackage[mathlines]{lineno} 
\usepackage{amsmath,amssymb,amsthm,amsfonts,amsbsy,latexsym,dsfont,color} 
\usepackage[numeric,initials,nobysame]{amsrefs} 
\usepackage[foot]{amsaddr}
\usepackage{upref,halloweenmath}

\usepackage[utf8]{inputenc}
\usepackage[english]{babel}

\usepackage[textsize=tiny]{todonotes}
\usepackage{regexpatch}
\makeatletter
\xpatchcmd{\@todo}{\setkeys{todonotes}{#1}}{\setkeys{todonotes}{inline,#1}}{}{}
\makeatother

\usepackage{etoolbox}          
\newcommand*\linenomathpatch[1]{%
  \cspreto{#1}{\linenomath}%
  \cspreto{#1*}{\linenomath}%
  \csappto{end#1}{\endlinenomath}%
  \csappto{end#1*}{\endlinenomath}%
}
\newcommand*\linenomathpatchAMS[1]{%
  \cspreto{#1}{\linenomathAMS}%
  \cspreto{#1*}{\linenomathAMS}%
  \csappto{end#1}{\endlinenomath}%
  \csappto{end#1*}{\endlinenomath}%
}

\expandafter\ifx\linenomath\linenomathWithnumbers
  \let\linenomathAMS\linenomathWithnumbers
  \patchcmd\linenomathAMS{\advance\postdisplaypenalty\linenopenalty}{}{}{}
\else
  \let\linenomathAMS\linenomathNonumbers
\fi

\linenomathpatch{equation}
\linenomathpatchAMS{gather}
\linenomathpatchAMS{multline}
\linenomathpatchAMS{align}
\linenomathpatchAMS{alignat}
\linenomathpatchAMS{flalign}

\makeatletter
\patchcmd{\mmeasure@}{\measuring@true}{
  \measuring@true
  \ifnum-\linenopenaltypar>\interdisplaylinepenalty
    \advance\interdisplaylinepenalty-\linenopenalty
  \fi
  }{}{}
\makeatother

\usepackage{enumerate}

\newenvironment{enumeratea}{\begin{enumerate}[\upshape (a)]}{\end{enumerate}}

\newtheorem{thm}{Theorem}[section]
\newtheorem{lem}[thm]{Lemma}
\newtheorem{cor}[thm]{Corollary}

\newtheorem{rem}[thm]{Remark}

\newtheorem{ass}[thm]{Assumption}

\renewcommand{\le}{\leqslant} 
\renewcommand{\ge}{\geqslant}

\newcommand{\ra}{\rangle}
\newcommand{\la}{\langle}

\newcommand{\eps}{\varepsilon}

\newcommand{\norm}[1]{\left\Vert#1\right\Vert}
\newcommand{\abs}[1]{\left\vert#1\right\vert}

\newcommand{\ie}{\emph{i.e.,}}

\def\qed{ \hfill $\blacksquare$}  
\let\ga=\alpha \let\gb=\beta \let\gc=\gamma  
 \let\gh=\eta    \let\gl=\lambda        \let\go=\omega   \let\gs=\sigma \let\gt=\tau 
  \let\gz=\zeta
\let\gC=\Gamma \let\gD=\Delta  \let\gL=\Lambda 
         \let\gS=\Sigma  
                                
\newcommand{\cE}{\mathcal{E}}
\newcommand{\cG}{\mathcal{G}}

\newcommand{\cO}{\mathcal{O}}
\newcommand{\cQ}{\mathcal{Q}}

\newcommand{\vone}{\mathbf{1}}

\newcommand{\vR}{\mathbf{R}}

\newcommand{\vZ}{\mathbf{Z}}

\newcommand{\vq}{\mathbf{q}}

\newcommand{\mvW}{\boldsymbol{W}}
\newcommand{\mvZ}{\boldsymbol{Z}}

\newcommand{\mvq}{\boldsymbol{q}}

\newcommand{\mvx}{\boldsymbol{x}}\newcommand{\mvy}{\boldsymbol{y}}

\newcommand{\mvgs}{\boldsymbol{\sigma}}\newcommand{\mvgt}{\boldsymbol{\tau}}

\newcommand{\mveta}{\boldsymbol{\eta}}

\newcommand{\dC}{\mathds{C}}

\newcommand{\dN}{\mathds{N}}

\newcommand{\dR}{\mathds{R}}

\newcommand{\sP}{\mathscr{P}}

\newcommand{\subc}{\emph{sub-critical}}
\newcommand{\crit}{\emph{critical}}
\newcommand{\supc}{\emph{super-critical}}

\DeclareMathOperator{\E}{\mathds{E}}
\DeclareMathOperator{\pr}{\mathds{P}}

\DeclareMathOperator{\var}{Var}
\DeclareMathOperator{\cov}{Cov}

\DeclareMathOperator{\tr}{Tr} 
  
\DeclareMathOperator{\diag}{diag}

\DeclareMathOperator{\N}{N}

\newcommand{\wmvgs}{\widetilde{\mvgs}}

\newcommand{\hh}{\hat{h}}

\begin{document}

\title[Spin Glass models under Weak External Field]{Mean Field Spin Glass Models under Weak External Field}
\author[Dey]{Partha S.~Dey}
\author[Wu]{Qiang Wu}
\address{Department of Mathematics, University of Illinois at Urbana-Champaign, 1409 W Green Street, Urbana, Illinois 61801}
\email{$\{$psdey,qiangwu2$\}$@illinois.edu}
\date{\today}
\begin{abstract}
	We study the fluctuation and limiting distribution of free energy in mean-field Ising spin glass models under weak external fields. We prove that at high temperature, there are three sub-regimes concerning the strength of external field $h \approx \rho N^{-\alpha}$ with $\rho,\alpha\in (0,\infty)$. In the super-critical regime $\alpha < 1/4$, the variance of the log-partition function is $\approx N^{1-4\alpha}$. In the critical regime $\alpha = 1/4$, the fluctuation is of constant order but depends on $\rho$. Whereas, in the sub-critical regime $\alpha>1/4$, the variance is $\Theta(1)$ and does not depend on $\rho$. We explicitly express the asymptotic mean and variance in all three regimes and prove Gaussian central limit theorems. 
	
	Our proofs mainly follow two approaches. One generalizes quadratic coupling and Guerra's interpolation scheme for Gaussian disorder, extending to other spin glass models. This approach can establish CLT at high temperature by proving a limiting variance result for the overlap. The other is a combination of cluster expansion for general symmetric disorders and multivariate Stein's method for exchangeable pairs. For the zero external field case,  cluster expansion was first used in the seminal work of Aizenman, Lebowitz, and Ruelle~(Comm.~Math.~Phys.~112 (1987), no.~1, 3--20). It was believed that this approach does not work if the external field is present. We show that if the external field is present but not too strong, it still works with a new cluster structure. In particular, we prove the CLT up to the critical temperature in the Sherrington-Kirkpatrick (SK) model when $\alpha \ge 1/4$. 
	
	We further address the generality of this cluster-based approach. Specifically, we give limiting results for the multi-species and diluted SK models. We believe that this approach will shed new light on how the external field affects the general spin glass system. We also obtain explicit convergence rates when applying Stein's method to establish the CLTs.
\end{abstract}

\maketitle
\setcounter{tocdepth}{1}\tableofcontents

\section{Introduction and Main Results}\label{sec:intro}
Spin glass model was initially introduced to study the strange magnetic alloy behaviors of disordered materials. One of the most notable models is the Sherrington-Kirkpatrick (SK) mean-field model~\cite{SK72}. Over the last few decades, a large amount of research has been conducted to understand the SK model in the mathematics and physics community. The fundamental question is to compute the limiting free energy. In the '80s, Parisi~\cites{Par79, Par80} proposed his famous replica symmetry breaking theory, suggesting a beautiful variational formula for the limiting free energy at any temperature. Until 2006, Talagrand gave the first rigorous proof of Parisi's prediction in his tour de force work~\cite{Tal06}. However, Parisi's formula only gives the first-order asymptotic of free energy. The second-order asymptotic picture, such as fluctuation order, and limit theorems of free energy, is only partially known. This work mainly focuses on studying the fluctuation problems in a broad class of mean-field spin glass models. 

In the seminal work~\cite{ALR87}, Aizenman, Lebowitz, and Ruelle first established fluctuation of free energy for the SK model~\cite{SK72} in the whole high-temperature regime at zero external field. They showed that the log-partition function has a constant order variance and admits a Gaussian distribution limit after proper centering. The key technique was based on the cluster expansion approach. This approach works for a general symmetric disorder. However, it has been believed that it does not work when the external field is present; see the discussion in~\cite{Tal11a}*{Section 1.14} and the related literature~\cite{Kos06,KL05}. Instead, by using the Gaussian interpolation scheme,  Guerra and Toninelli~\cite{GT02} later proved that the fluctuation order becomes linear when a positive external field is present in the model and derived a central limit theorem after appropriate centering and scaling. Nevertheless, this result can only hold at very high temperatures. It is folklore that understanding spin glass models with a non-zero external field is more challenging, especially when one proves results up to the critical temperature. Afterward, there were different approaches developed to prove the same fluctuation results, see~\cites{Tin05, CN95}. Although much progress has been made on the fluctuation problems in spin glass models, there is still a gap in understanding the transitional behavior between no external field and a positive external field case.
Parisi's formula gives the thermodynamic limit in a variational form and is Lipschitz w.r.t.~the external field. It neither suggests fluctuation results nor gives transitional behavior under external fields. 
More importantly, it has never been clear why the cluster expansion approach does not work in the presence of an external field. 

This paper complements the gap by proving Gaussian central limit theorem for the free energy with explicit mean and variance when the external field is weak. The weak external field will decay to zero at various rates in the infinite volume limit but not exactly zero for the finite system. We summarize the detailed contributions as follows. First, we present several approaches to establish the fluctuation results for a broad class of spin glass models. When the external field is weak enough but still present, the results hold up to the critical temperature. Second, one of the approaches successfully extends the cluster expansion method used in the seminal work~\cite{ALR87} to models with an external field. In particular, we found that the cluster expansion approach still works, and a new cluster structure emerges. It provides a strong evidence against the common belief in~\cite{Tal11a}*{Section 1.14} and~\cite{Kos06,KL05}. More importantly, this gives a new combinatorial perspective on how the external field affects the spin-glass system. We believe that this tool will also be useful in analyzing many other statistical physics models. In particular, we applied this analysis in the diluted SK model and multi-species SK model to highlight its generality. Besides that, when establishing the central limit theorem, we creatively adapt the multivariate exchangeable pair Stein's method to the cluster structures, enabling us to obtain an explicit convergence rate for the normal approximation. This is also of independent interest. In other existing works~\cite{ALR87,Tin05,CN95,GT02}, the moment method, a classical tool for establishing CLT, can not give explicit convergence rates. 
Finally, using the analytical approach based on the characteristic framework, we prove the CLT at high enough temperature with Gaussian disorder. The major difficulty here is establishing a limiting variance result for the overlap, which is of independent interest. We achieve this by extending the classical quadratic coupling argument of Guerra-Toninelli~\cite{GT02} along with a second moment computation. 

Let us first recall the definition of spin glass models and introduce the problem of interest.

\subsection{Definitions}\label{ssec:model}
First, we define the celebrated Sherrington--Kirkpatrick model and some generalizations.

\subsubsection{SK Model}

Fix a large integer $N>0$ and consider a configuration $\mvgs :=(\gs_1, \gs_2,\cdots, \gs_N) \in \gS_N:=\{-1,+1\}^N$. The Hamiltonian for the SK model is given by
\begin{align}\label{def:H_N}
	H_N(\mvgs) = \frac{\gb}{\sqrt{N}} \sum_{i<j }J_{ij} \gs_i\gs_j + h\sum_{i=1}^N \gs_i
\end{align}
where $(J_{ij})$'s are i.i.d.~random variables with mean zero and variance one. Here, $h\ge0$ is the external field, and $\gb>0$ is the inverse temperature. The structure function $\xi(x):\dR \to \dR$ is defined as $\xi(x)=\frac12\gb^2x^2$ so that
\begin{align}\label{eq:structure}
	\frac1N\cov \left (H_{N}(\mvgs),H_{N}(\mvgt) \right) = \frac12\gb^2R_{1,2}^2-\frac{\gb^2}{2N}\approx \xi(R_{1,2})
\end{align}
where $R_{1,2}:=R(\mvgs,\mvgt)=\frac{1}{N} \sum_{i=1}^N \gs_i \gt_i$ is known as the \emph{overlap} in SK model. Equality in~\eqref{eq:structure} can be obtained by adding an independent $\N(0,\gb^2/2)$ term to the Hamiltonian, which does not affect the Gibbs measure. Classically one takes $J_{12}\sim\N(0,1)$. The other generalized mean field spin glass models, such as pure $p$-spin or mixed models can be defined by specifying the structure function $\xi(x)$ as $\frac12\sum_{p\ge 2}\gb_p^2 x^p$ for appropriate choices of $(\gb_p)_{p\ge 2}$.

\subsubsection{Multi-Species SK model}
This Multi-species SK model (MSK) is an inhomogeneous extension of the SK model introduced by Barra~\emph{et.~al.}~\cite{BCPMT15}. Compared to the SK model, the given $N$ spins are partitioned into $m>1$ species of size $N_s,s=1,2,\ldots,m$, respectively. Assume that the number of species $m$ is fixed and finite. Specifically,
$
	I:= \{1,2,\ldots, N\} = \bigcup_{s=1}^m I_s, \text{ where } I_s \cap I_t = \emptyset \text{ for } s\neq t.
$
Given a spin configuration $\mvgs=(\gs_1,\gs_2,\ldots, \gs_N) \in \gS_N$, the Hamiltonian is defined as
\begin{align*}
	H_N(\mvgs):=\frac{\gb}{\sqrt{N}}\sum_{s,t=1}^m\sum_{i<j,i\in I_s, j\in I_t} J_{ij}\gs_i\gs_j + h \sum_{i=1}^N \gs_i,
\end{align*}
where $J_{ij}$ are independent centered random variables, and $J_{ij} \sim \N(0,\gD_{st}^2)$ for $i \in I_s, j\in I_t$. We also denote the relative size of species $s$ by $\gl_s: = {\abs{I_s}}/{N}, s=1,2,\ldots,m$. Define the interaction matrix $\gD:=((\gD_{st}^{2}))_{s,t=1}^{m}$ and the species density matrix $\gL:=\diag(\gl_{1},\gl_{2},\ldots,\gl_{m})$.
One can easily compute the structure function for the MSK model as
\begin{align*}
	\xi(\mvx) := \frac12\gb^2\cQ(\mvx) = \frac12\gb^2 \mvx^T\gL \gD^2 \gL \mvx,\quad \mvx\in\dR^m;
\end{align*}
with covariance structure for the Gaussian field $(H_N(\cdot))$ given by $\xi(\vR_{1,2})$. Here
\begin{align*}
	\vR_{1,2} := (R_{1,2}^{(1)}, R_{1,2}^{(2)}, \ldots, R_{1,2}^{(m)})\text{ with }R_{1,2}^{(s)}:= \frac{1}{\abs{I_s}} \sum_{i \in I_s} \gs_i \gt_i
\end{align*}
being the overlap parameter in species $s$ for $s=1,2,\ldots,m$.

For the above models, the associated free energy is defined by
\begin{align}\label{eq:fren}
	F_N(\gb,h) := \frac{1}{N}\log Z_N(\gb,h) \text{ where } Z_N(\gb,h):= \sum_{\mvgs \in \gS_N} \exp(H_N(\mvgs))
\end{align}
and the corresponding Gibbs measure by
\begin{align*}
	d G_{N,\gb,h}(\mvgs) = Z_N(\gb,h)^{-1}\exp(H_N(\mvgs)) d\mvgs,
\end{align*}
where $d \mvgs$ denotes the uniform measure on $\gS_N$. For convenience, we will use $\la f \ra_{\gb,h}$ to denote the average of $f: \gS_N \to \dR$ with respect to the Gibbs measure $d G_{N,\gb,h}(\mvgs)$. For the rest of this article, the weak external field $h$ will be of the form stated below in Assumption~\ref{ass:h}.

\begin{ass}\label{ass:h}
	We have $h:=h_N= \rho N^{-\ga}, \text{ where } \rho, \ga \in (0,\infty)$ are positive constants. We will use $\hh:=\tanh(h)$.
\end{ass}

We also drop the dependence on $N,\gb,\ga, \rho$ for notational convenience. The problem of interest is the fluctuation of the free energy $F_N(\gb,h)$ in the above spin glass models under Assumption~\ref{ass:h}. In particular, $\ga=0$ corresponds to the classical spin glass models with a positive external field, while $\ga=\infty$ reduces to the zero external field case. Before stating our main results, let us briefly overview the related results in the $\ga\in\{0,\infty\}$ cases and the spherical SK model with a weak external field.

\subsection{Related Results}\label{ssec:relate}
For the Ising SK model, the fluctuation of free energy without an external field has long been known to be Gaussian in the high-temperature regime. This result was first proved by Aizenman, Ruelle, and Lebowitz~\cite{ALR87} using the cluster expansion technique in 1987. Later there were other different proofs of the same result: stochastic analysis approach~\cite{CN95}, and moment method in~\cite{Tal11b}. The fluctuation problem in low temperature and critical temperature with no external field is still wide open. The fluctuation order is not even clear; we refer to the most recent progress in~\cites{Ch19, CL19}. When the external field is present, \ie\ $h>0$, Guerra and Toninelli proved that the fluctuation is Gaussian at a very high temperature. The most general result is due to the Chen, Dey, and Panchenko~\cite{CDP17}, who proved a Gaussian central limit theorem for the free energy in the setting of mixed $p$-spin models at all non-zero temperatures.

For the spherical (soft spin) SK model, where the spins can take continuous values on a high dimensional sphere, Baik and Lee~\cite{BL16} proved that the free energy converges to a Gaussian distribution in the high-temperature regime, and Gaussian orthogonal Ensemble (GOE) Tracy-Widom distribution in the low-temperature regime. In particular, at zero temperature ($\gb=\infty$), the free energy, known as the ground state energy, corresponds to the largest eigenvalue of a GOE random matrix. The fluctuation is GOE Tracy-Widom by classical random matrix literature. The key ingredient in their proof is a contour integral representation of the partition function. Most recently, Landon and Sosoe extended these methods to the critical and low temperatures; we invite interested readers to the following works~\cites{LS19, NS19, Lan20} and references therein for the progress in spherical models.

For the fluctuation problem under a weak external field, in the setting of spherical models, there is a physics rigorous level paper~\cite{BCLW21}, where the authors study the fluctuation of various objects, including the free energy and overlap, among others. A contemporary work by Landon and Sosoe~\cite{LS20} also proved part of these results in a mathematically rigorous way. These results successfully detect the correct threshold of the external field strength and the corresponding fluctuation order. The essential technique can be traced back to~\cite{BL16}, where a contour integral representation of the log-partition function is essential. However, using the steepest descent method to determine the asymptotics of the integral is still highly technical. As far as we know, this technique is restricted to the spherical SK model. There is another work of a similar fashion due to Belius~\emph{et.~al.}~\cite{BCNS21}, who uses a geometric approach to find a threshold for the positive external field. More specifically, they utilized the Kac-Rice formula to compute the number of critical points of the Hamiltonian function. They found that the energy landscape is trivial if the external field is strong enough, \ie~there are only two critical points. Below the threshold, the Hamiltonian has exponentially many critical points.
In summary, all of the above results are for spherical spin glass models where many tools from random matrix theory can be utilized.

The hypercube configuration space has less symmetry than the sphere and is discrete in nature. Thus the fluctuation problem under external fields becomes more challenging for the Ising spin case. Many powerful tools such as contour integral representation and complexity-based approaches can not be applied directly in the discrete setting. In this work, we develop several different approaches that are more robust in the sense that they can not just work in the Ising spin case but also in the spherical and many other spin glass models.

\subsection{Main Results}\label{ssec:results}
Before stating the main theorem, we discuss some heuristic ideas on how to find the critical threshold of $\ga$ in $h\approx N^{-\ga}$. For Gaussian disorder, it is well-known that at high temperatures the overlap $R_{12}$ concentrates on some point $q\in[0,1]$, satisfying the following fundamental equation
\begin{align}\label{eq:funda}
	q=\E\tanh^2(\gb\sqrt{q}\eta+h) \text{ where }\eta \sim \N(0,1).
\end{align}
The uniqueness of solution to~\eqref{eq:funda} is guaranteed by Latala-Guerra Lemma~\cite{Tal11a}*{Proposition 1.3.8}. Using the above equation, we have the following lemma characterizing the relation of $q,\gb,h$. The proof is given in Section~\ref{pf:auxlem}.

\begin{lem}\label{lem:1}
	Let $q$ satisfy the fixed point equation~\eqref{eq:funda} with $\gb<1$. Then
	\begin{align*}
		\frac{h^2}{1-\gb^2} -\frac{2h^4}{(1-\gb^2)^3}\le q \le \frac{h^2}{1-\gb^2}.
	\end{align*}
\end{lem}
From this lemma, one can see that $q \approx {h^2}/(1-\gb^2)$ for $h$ small. On the other hand, in the smart path interpolation~\eqref{eq:interpo} of Section~\ref{sec:char}, the SK model is supposed to behave similarly as the decoupled model ($H_{N,0}$ in~\eqref{eq:interpo}) in the infinite volume limit. Lemma~\ref{lem:2} suggests that the variance order of log-partition function in the decoupled model is $Nq^2\approx Nh^4\approx N^{1-4\ga}$. This heuristic tells that the critical threshold of $\ga$ should be $1/4$. In later sections, we will see that this intuition is indeed correct from several rigorous proofs.

Now let us state the main theorem for the SK spin glass model as follows.
\begin{thm}[SK model under weak external field]\label{thm:main}
	Let $h$ satisfy Assumption~\ref{ass:h}, and $J_{ij}, 1\le i<j\le N$ have i.i.d.~symmetric distribution with variance one and finite sixth moment. We have the following results for the SK model with disorder $(J_{ij})_{i<j}$, inverse temperature $\gb$, and external field $h$.
	\begin{enumeratea}
		\item \textsc{\bfseries Sub-critical regime}: If $\ga>1/4$ and $\gb<1$, then
		 \begin{align*}
			 & \log Z_N(\gb,h) - N\log(2\cosh h) - \frac14(N-1)\gb^2 \\
			 & \qquad\implies \N\left(-\frac12v_1^2 - \frac{\gb^4}{24}\E (J_{12}^4),\; v_1^2+\frac{\gb^4}{8}\E(J_{12}^4-1)\right)
		 \end{align*}
		 in distribution as $N\to \infty$ where
		 \begin{align*}
			 v_1^2 := -\frac12\log(1-\gb^2)- \frac{\gb^2}{2} - \frac{\gb^4}4.
		 \end{align*}
		\item \textsc{\bfseries Critical regime}: If $\ga = 1/4$ and $\gb <1 $, then
		 \begin{align*}
			 & \log Z_N(\gb,h) - N\log(2\cosh h) - \frac14(N-1)\gb^2 \\
			 & \qquad\implies \N\left(-\frac12(v_1^2 + \rho^4 v_2^2) - \frac{\gb^4}{24}\E (J_{12}^4),\; v_1^2 + \rho^4 v_2^2 +\frac{\gb^4}{8}\E( J_{12}^4-1)\right)
		 \end{align*}
		 in distribution as $N\to \infty$ where
		 \begin{align*}
			 v_2^2 := \frac{\gb^2}{2(1-\gb^2)}.
		 \end{align*}

		\item \textsc{\bfseries Super-critical regime}: Assuming Gaussian disorder, there exists $\gb_0\le 1$ such that if $\ga<1/4$ and $\gb < \gb_0$, then
		 \begin{align*}
			 & \frac1{\sqrt{Nh^4}} \biggl(\log Z_N(\gb,h) - N\E\log (2 \cosh(\gb\sqrt{q}\eta+h)) - \frac14N \gb^2(1-q)^2 \biggr) \\
			 & \qquad\implies \N(0,v_2^2)
		 \end{align*}
		 in distribution as $N\to \infty$.
	\end{enumeratea}
\end{thm} 

\begin{rem}[Optimality of moment assumptions]
	We need finite sixth moment for $J_{ij}$ to use Stein's method for $J_{ij}^{2}$. Using standard arguments and H\"older inequality, one can reduce the moment condition from sixth moment to $(4+\eps)$--th moment for any $\eps>0$. However, finite fourth moment is required to get a Gaussian CLT for $\sum_{i<j} \log\cosh(\gb J_{ij}/\sqrt{N})$.
\end{rem}

\begin{rem}[Optimality of the bounds on $\beta$]
	The cases $\ga=0$ and $\ga=\infty$ correspond to the SK model with a positive external field and no external field. The fluctuation results are already known in the classical literature (See~\cite{Tal03}) and should hold for all $\gb<\gb_{c}=1$. We expect $\gb_0=1$ for all symmetric i.i.d.~disorders with variance $1$ in the \supc~regime, but to rigorously prove this is a challenging question. This also relates to the famous Almeida-Thouless transition conjecture; see the recent progress in~\cite{Chen21AT} and references therein.
\end{rem} 

\begin{rem}[Disparity in the means]
	In all three different regimes, the asymptotic mean is of the same form, even if they look different. In the critical regime, using $h=\rho N^{-1/4}$, the relation $q \approx {h^2}/(1-\gb^2)$ and $\log\cosh x \approx x^{2}/2 +O(x^{4})$, we get
	\begin{align*}
		 & \left(N\E\log\cosh(\gb\sqrt{q}\eta+h) + \frac12N \gb^2(1-q)^2\right) -\left( N\log\cosh h+\frac14(N-1)\gb^{2}\right) \\
		 & \qquad = \frac12N (\gb^{2}q+h^{2}) - \frac12(N \gb^2q+Nh^{2}) +O(1) = O(1).
	\end{align*}
\end{rem} 

Now we state a similar result for the MSK model on $N$ spins with $m$ species of size $N_{s}:=N(\gl_{s}+o(1)),s=1,2,\ldots,m$, respectively and interaction matrix $\gD^{2}=((\gD_{st}^{2}))_{s,t=1}^{m}$. Recall that $\gL$ is the diagonal matrix with $s$--th diagonal entry $\gl_{s}$ for $s=1,2\ldots,m$. In the zero external field case the critical $\gb$ is given by $\gb_c:=\norm{\gL^{1/2}\gD^{2}\gL^{1/2}}^{-1/2}$ (see~\cite{DW20}).

Similar to the classical SK model, we now consider the MSK model with external field $h=\rho N^{-\ga}$ for some $\rho,\ga\in(0,\infty)$.
Define the matrix $$\gC_\gb:=I-\gb^2 \gD^2\gL,$$ which is invertible in the high temperature regime $\gb<\gb_{c}$. In the Gaussian disorder case, it is easy to see that the $\mvq$ vector satisfies
\begin{align*}
	\norm{\mvq - h^2\cdot \gC_\gb^{-1}\vone }= O(h^4).
\end{align*}
Let $s(i)$ denote the species containing vertex $i$. We assume that the disorder $(J_{ij})_{i<j}$ satisfies the following assumption.
\begin{ass}\label{ass:J}
	We have $J_{ij}:= \gD_{s(i)s(j)}\cdot \hat{J}_{ij}$, where $(\hat{J}_{ij})_{i<j}$ are i.i.d.~from a symmetric distribution with mean zero, variance one and finite sixth moment.
\end{ass}

\begin{thm}[MSK model under weak external field]\label{thm:main2}
	Suppose that Assumptions~\ref{ass:h} and \ref{ass:J} hold.
	For the MSK model, we have the following results.
	\begin{enumeratea}
		\item \textsc{\bfseries Sub-critical regime}: If $\ga>1/4$ and $\gb<\gb_{c}$, then
		 \begin{align*}
			 & \log Z_N(\gb,h) - N\log(2\cosh h) - \frac14N\gb^2\cQ(\vone) + \frac{\gb^2}{4}\tr(\gD^2\gL) \\
			 & \qquad\implies \N\left(-\frac12v_1^2 - \frac{\gb^4}{24}\E (\hat{J}_{12}^4)\cQ_2(\vone),\; v_{1}^{2}+ \frac{\gb^4}{8}\E(\hat{J}_{12}^4-1) \cQ_2(\vone)\right)
		 \end{align*}
		 in distribution as $N\to \infty$ where
		 \begin{align*}
			 v_1^2
			 & := -\frac12\log\det(\gC_\gb) -\frac{\gb^2}2\tr(\gD^2\gL) - \frac{\gb^4}4 \cQ_{2}(\vone), \\
			 \cQ(\vone) & :=\vone^T\gL\gD^2\gL\vone,\qquad
			 \cQ_2(\vone) :=\vone^T\gL(\gD^2\circ \gD^2)\gL\vone,
		 \end{align*}
		 and $\circ$ is the Hadamard product.
		\item \textsc{\bfseries Critical regime}: If $\ga = 1/4$ and $\gb <\gb_{c}$, then
		 \begin{align*}
			 & \log Z_N(\gb,h) - N\log(2\cosh h) - \frac14N\gb^2 \cQ(\vone) + \frac{\gb^2}{4}\tr(\gD^2\gL) \\
			 & \quad\implies \N\left(-\frac12(v_1^2 + \rho^4 v_2^2) - \frac{\gb^4}{24}\E (\hat{J}_{12}^4)\cQ_2(\vone),\; v_{1}^{2}+ \rho^4 v_2^2+\frac{\gb^4}{8}\E(\hat{J}_{12}^4-1) \cQ_2(\vone)\right)
		 \end{align*}
		 in distribution as $N\to \infty$ where
		 \begin{align*}
			 v_2^2 & := \frac{\gb^2}2 \vone^T\gL\gD^2\gL\gC_\gb^{-1}\vone.
		 \end{align*}

		\item \textsc{\bfseries Super-critical regime}: Assuming Gaussian disorder, there exists $\gb_0\le \gb_{c}$ such that if $\ga<1/4$ and $\gb < \gb_0$, then
		 \begin{align*}
			 & \frac1{\sqrt{Nh^4}} \biggl(\log Z_N(\gb,h) -
			 \sum_{s=1}^m N_s \E \log 2\cosh(\gb \eta \sqrt{(\gD^2\gL \vq)_s} +h) - \frac14N\gb^2\cQ(\vone -\vq)
			 \biggr) \\
			 & \qquad\implies \N(0,v_2^2)
		 \end{align*}
		 in distribution as $N\to \infty$.
	\end{enumeratea}
\end{thm}

This matches the SK results in the $m=1$ species case.
Note that, under Assumption~\ref{ass:J} one can easily check that
\begin{align*}
	\sum_{i<j}\log\cosh(\gb J_{ij}/\sqrt{N}) - &
	\frac14N\gb^2\cQ(\vone) + \frac{\gb^2}{4}\tr(\gD^2\gL) \\
	 & \implies \N\left( - \frac{\gb^4}{24}\E (\hat{J}_{12}^4)\cQ_2(\vone),\; \frac{\gb^4}{8}\E(\hat{J}_{12}^4-1) \cQ_2(\vone)\right)
\end{align*}
in distribution as $N\to\infty$.

The proof of Theorem~\ref{thm:main2} is precisely the same as the proof for the classical SK model. However, the cluster expansion approach for the general MSK model is quite involved since the underlying graph is a complete weighted graph.

Instead of proving the the theorem for general MSK model, we focus on the proof for the bipartite SK model with two species with
\begin{align*}
	\gD^{2}=\begin{pmatrix}0& 1\\1&0\end{pmatrix}\text{ and }\gL=\begin{pmatrix}p_{1}& 0\\0&p_{2}\end{pmatrix}.
\end{align*}
The phase transition happens at $\gb_{c}:=(p_1p_2)^{-1/4}$ in the zero external field case. In particular, we have, for Theorem~\ref{thm:main2},
\begin{align*}
	v_{1}^{2} & =-\frac12\biggl(\log(1-\gb^{4}p_1p_2)+\gb^{4}p_1p_2 \biggr), \\
	v_{2}^{2} & =\frac12\cdot \frac{\gb^4 + 2\gb^{2}}{1-\gb^{4}p_1p_2}\cdot p_1p_2, \\
	\cQ(\vone) & =\cQ_{2}(\vone)=2p_1p_2 \text{ and } \tr(\gD^{2}\gL)=0.
\end{align*}
The proof for the bipartite case is given in Section~\ref{sec:extension}, where we also state the similar results for the diluted SK model and sketch the proof idea.  

For the above theorems, depending on whether the disorder coupling $(J_{ij})_{1\le i<j\le N}$ are Gaussian or not, we have several different approaches to prove Theorem~\ref{thm:main}. In the following subsection~\ref{ssec:proof}, we will give a summary of the features of different approaches and some heuristic details. Here, we mention that the Gaussian approach can also be extended to the spherical models giving a soft proof for the Gaussian CLT in a high enough temperature regime.  

\subsection{A word on the proof}\label{ssec:proof}

In the $\ga \ge 1/4$ regime, Theorem~\ref{thm:main} is proved by using two different approaches. One is more analytical, and the disorder coupling needs to be i.i.d.~\emph{Gaussian} random variables. The other is more combinatorial, and the Gaussianity of disorder coupling is relaxed to general symmetric distribution with appropriate moment assumptions. In the \supc~regime $\ga < 1/4$, we can only use the analytic idea to prove the central limit theorem. The critical step is based on controlling the quenched fluctuation of overlap. The strength of the analytical approach is that it can be easily extended to other variants of spin glass models, such as the pure even $p$-spin model and multi-species SK model. However, the weakness is that the central limit theorem can only be proved at very high temperatures. The combinatorial technique can work up to the critical temperature, but in the \supc~regime, the combinatorial structures become quite involved, and it does not work anymore. However, for $\ga \ge 1/4$, we believe that this approach may also be adapted to other spin glass models, but it could be highly technical. Using this approach in the diluted SK model with a weak external field, there will be some extra contributions to the partition function; in particular, if the external field is strong enough, some non-Gaussian fluctuation is expected.

\subsubsection{Gaussian Disorder Case}
In this proof, the framework is based on the smart path interpolation and characteristic approach. The idea was also used in~\cite{DW20} for proving the central limit theorem of free energy in the multi-species SK model. Gaussianity of the disorder is essential to the proof as Gaussian integration by parts was utilized multiple times. The critical step is to control the quenched fluctuation of the centered overlap (c.f.~\eqref{eq:deriv}), \ie\
$
	N \la (R_{12}-q)^2 \ra_t,
$
where $\la \cdot \ra_t$ is average under the Gibbs measure with the interpolated \emph{Hamiltonian} $H_{N,t}(\cdot)$ in~\eqref{eq:interpo}. It turns out that in the \supc~regime, it is sufficient to prove that
\begin{align*}
	N \la (R_{12}-q)^2 \ra_t \ll \sqrt{Nh^4},
\end{align*}
which can be done by an easy adaptation of Latala's argument~\cite{La02}. In the \subc~and \crit~regime, we need a stronger result of the form, for $t\in(0,1)$
\begin{align*}
	N \la (R_{12}-q)^2 \ra_t \to c_t:= 1/(1-\gb^2t)\text{ in } L^1 \text{ as } N\to\infty.
\end{align*}
Here the value of $c_t$ is obtained from the $h=0$ case. To achieve this, we need to control the following objects,
\begin{align*}
	U_{N,t} & := \frac{1}{(\E Z_{N,t}(\gb,h))^2} \sum_{\mvgs,\mvgt} (N(R_{12}-q)^2 -c_t)\cdot e^{H_{N,t}(\mvgs)+H_{N,t}(\mvgt)} \\
	\text{ and }\qquad
	 & |\E Z_{N,t}(\gb,h)|^2/Z_{N,t}(\gb,h)^2.
\end{align*}

The behavior of $U_{N,t}$ can be analyzed via the asymptotic of a Gaussian integral. The second object is controlled by deriving upper tail estimates of the log-partition function. Along with that, we also generalize the quadratic coupling method originally used in~\cite{GT04} for $h=0$.

Beyond the above characteristic approach framework, we give another short proof for the \subc~regime under Gaussian disorders. This proof works up to the critical temperature. The idea is simple: we prove that the difference between the centered log-partition function with a weak enough external field and with zero external field converges to $0$ in probability. The central limit theorem immediately follows from the classical result for zero external fields. This explains why in the~\subc~regime, the variance is the same as the zero external field case. The proof is based on the second-moment method. Another important ingredient is the Gaussian trick, also known in the literature as Hubbard--Stratonovich transformation~\cite{Hub59},
\begin{align}\label{eq:HS-trans}
	\exp\left(a/2\right) = \E\exp(\sqrt{a}\cdot \eta),
\end{align}
for any complex number $a\in \dC$ and $\eta\sim\N(0,1)$. It is also used in the Gaussian characteristic framework.

\subsubsection{General Disorder Case}
For general symmetric disorders, the combinatorial approach known as cluster expansion is a useful tool in the physics community. It was first used in the SK model without external field by Aizenman, Lebowitz, and Ruelle~\cite{ALR87}. Since then, it has been believed that this approach could not work when the external field is present; see the discussion in~\cite{Tal11a}*{Section 1.14}. This work shows that it can work if the external field is not too strong and how it affects the free energy. The first step is the following decomposition
\begin{align}\label{def:Zdecomp}
	Z_N(\gb,h) = (2\cosh h)^N\prod_{i<j} \cosh(\gb J_{ij}/\sqrt{N}) \cdot \E_{\mvgs} \prod_{i<j} \bigl(1+ \gs_i\gs_j \tanh(\gb J_{ij}/\sqrt{N})\bigr)
\end{align}
where $\mvgs=(\gs_1,\gs_2,\ldots,\gs_N)$ are i.i.d.~$\pm1$--valued random variables with mean $\hh=\tanh h$. Expanding the $\E_{\mvgs}$ part, we obtain
\begin{align}\label{eq:intro-expan}
	\E_{\mvgs} \prod_{i<j} \bigl(1+ \gs_i\gs_j \tanh({\gb J_{ij}}/{\sqrt{N}})\bigr)
	= \sum_{\gC \subseteq \cE_N} \hh^{|\partial\gC|}\cdot \go(\gC),
\end{align}
where $\cE_N$ is edge set for the complete graph with $N$ vertices, and $\abs{\partial \gC}$ is the number of odd degree vertices in $\gC$ and
\begin{align*}
	\go(\gC):=\prod_{e\in\gC} \tanh(\gb J_{e}/\sqrt{N}).
\end{align*}
An important fact in~\eqref{eq:intro-expan} is that, in the regimes $\ga\ge 1/4$, the contribution of large graphs to the partition function decays exponentially fast in $\abs{\gC}$. This makes it possible to study the asymptotics of~\eqref{eq:intro-expan} combinatorically. In contrast, in the \supc~regime, very large graphs are important and counting the cycles and paths becomes intractable. Restriced to the finite graph case in~\eqref{eq:intro-expan}, a further reduction gives
\begin{align*}
	\E_{\mvgs} \prod_{i<j} \bigl(1+ \gs_i\gs_j \tanh({\gb J_{ij}}/{\sqrt{N}})\bigr)
	\approx
	\prod_{\gc: |\gc|\le m} (1+\go(\gc)) \prod_{p: |p|\le m } (1+\hh^{2}\cdot \go(p)),
\end{align*}
for $m$ large where $\gc, p$ denotes the loops and paths in the complete graph $K_{N}$. Then the following picture is immediate,
\begin{align}\label{eq:cluster-decomp}
	\log Z_N - & N\log(2\cosh h)
	\approx \sum_{e \in \cE_N} \log \cosh(\gb J_e/\sqrt{N}) \notag \\
	 & \qquad + \sum_{m\ge 3}\sum_{\gc\,:\,\abs{\gc}=m} \log (1+\go(\gc)) + \sum_{k\ge 2}\sum_{p\,:\,\abs{p}=k} \log (1+\hh^2 \cdot \go(p)).
\end{align}
Compared to~\cite{ALR87} and other works~\cites{ALS21,Ban20,BD21} built on it, the path cluster is a new structure that precisely suggests how the external field changes the behavior of the system from a combinatorial perspective. As we will prove later in Section~\ref{sec:stein}, the first sum and the other sums for finitely many values of $m,k$ in~\eqref{eq:cluster-decomp} are asymptotically independent Gaussian after appropriate centering; in particular
\begin{align*}
	\sum_{e \in \cE_N} \log \cosh(\gb J_e/\sqrt{N})
	 & \approx \frac{\gb^2}{4}(N-1) + \N\left( - \frac{\gb^4}{24}\E (J_{12}^4), \frac{\gb^4}{8}\E(J_{12}^4-1)\right), \\
	\sum_{m\ge 3}\sum_{\gc\,:\,\abs{\gc}=m} \log (1+\go(\gc))
	 & \approx \N\left( -\frac12\sum_{m\ge 3} \frac{\gb^{2m}}{2m},\sum_{m\ge 3} \frac{\gb^{2m}}{2m}\right)
	=\N\left(-\frac12v_1^2, v_1^2\right), \\
	\sum_{k\ge 2}\sum_{p\,:\,\abs{p}=k} \log (1+\hh^2 \cdot \go(p))
	 & \approx
	\N\left( -\frac12\sum_{k\ge 2} \frac{\rho^4\gb^{2k}}{2},\sum_{k\ge 2} \frac{\rho^4\gb^{2k}}{2}\right)
	=\N\left(-\frac12\rho^4v_2^2, \rho^4v_2^2\right)
\end{align*}
in distribution and are asymptotically independent. This explain the result in Theorem~\ref{thm:main}.

The final step is to prove these multivariate central limit theorems. In the zero external field case, the classical approach in~\cite{ALR87} is based on the moment method. However, a new path cluster is present in the current setting, and the moment method becomes quite involved to control all the mixed moments. Instead of carrying out those  complicated moment computations, we apply the following version of multivariate Stein's method. 

\begin{thm}[{\cite{RR09}*{Theorem 2.1}}]\label{thm:rr-mvstein}
	Assume that $(\mvW,\mvW')$ is an exchangeable pair of $\dR^d$ valued random vectors such that
	$
		\E \mvW = 0, \E \mvW\mvW^T = \gS
	$
	where $\gS \in \dR^{d \times d}$ is a symmetric positive definite matrix and $\mvW$ has finite third moment in each coordinate. Suppose further that
	\begin{align}\label{eq:stein-linear}
		\E(\mvW'-\mvW\mid \mvW) = -\gL \mvW\quad \text{a.s.},
	\end{align}
	for an invertible matrix $\gL$. If $\vZ$ has $d$-dimensional standard normal distribution, we have for every three times differentiable function $f$,
	\begin{align*}
		\abs{\E f(\mvW) - \E f(\gS^{1/2}\mvZ)} \le \frac{1}{4}\abs{f}_2A + \frac{1}{12}\abs{f}_3B
	\end{align*}
	where $\abs{f}_2:=\sup_{i,j}\norm{\partial_{x_ix_j} f}_\infty, \abs{f}_3:=\sup_{i,j,k}\norm{\partial_{x_ix_jx_k} f}_\infty$, $\gl^{(i)}: = \sum_{m=1}^d \abs{(\gL^{-1})_{m,i}}$,
	\begin{align*}
		A & := \sum_{i,j=1}^d \gl^{(i)} \sqrt{\var \E((W_i'-W_i)(W_j'-W_j)\mid \mvW)} \\
		\text{ and }
		B & := \sum_{i,j,k=1}^d \gl^{(i)} \E \abs{(W_i'-W_i)(W_j'-W_j)(W_k'-W_k)}.
	\end{align*}
\end{thm}

Stein's method is a state-of-the-art technique to prove normal convergence (and many others in general) for sums of weakly dependent random variables with an explicit convergence rate. Here we will not delve into the philosophy behind Stein's method or the proof of~Theorem~\ref{thm:rr-mvstein}. We refer the interested readers to the book~\cite{GCS11}.

In our case, compared to the moment method, Stein's method not only gives a more straightforward proof, but the analysis also presents a more transparent picture of the combinatorial structures. It suggests why the critical threshold of $\ga$ is $1/4$. Heuristically, the extra degree of freedom in the path clusters needs to be killed by the weak external field, whose strength has to be $\ga=1/4$. 

\subsection{Organization of the Paper}
This paper is structured as follows. In Section~\ref{sec:direct-pf}, based on the second moment method, we include a short proof for the \subc~regime in Theorem~\ref{thm:main} with \emph{Gaussian} disorder. It essentially shows that asymptotically the log-partition function under very weak external fields has no difference with the zero external field case. In Section~\ref{sec:char}, we use the smart path interpolation and characteristic approach to establish the framework to prove central limit theorems. Using the quadratic coupling method, we obtain the precise limit of $N\la (R_{12}-q)^2\ra_t$ and, using Latala's argument, establish the tightness of $N\la (R_{12}-q)^2\ra_t$. Combining, in Section~\ref{ssec:pf-gaussian} we present a proof of Theorem~\ref{thm:main} with \emph{Gaussian} disorder. After that, the cluster-based approach was expanded in Section~\ref{sec:cluster}, where we prove that there are three major sources of contribution to the free energy: independent sum part, loop clusters, and path clusters. In Section~\ref{sec:stein}, we use Stein's method to prove the multivariate central limit theorem; this gives proof for the \subc~and \crit~regime in Theorem~\ref{thm:main} under general symmetric distributed disorder having a finite sixth moment. The proof details are given in Section~\ref{sec:steincomp}. In Section~\ref{sec:extension}, we illustrate how the cluster-based approach in Section~\ref{sec:cluster} can easily be extended to other models establishing fluctuation results under weak external fields, such as the bipartite SK model (Section~\ref{sec:bsk}) and diluted SK model (Section~\ref{sec:dsk}). In Section~\ref{sec:open}, some open questions are addressed for future research. In particular, we discuss the possibility of applying the path and cycle counting technique to several models having very different structures than the SK case, such as the Ising Perceptron model and diluted $p$-spin model.

\section{Direct Proof in the Sub-Critical Case for Gaussian Disorder}\label{sec:direct-pf}
This section presents a direct yet straightforward proof in the \subc~regime. We show that the model has no difference with the $h=0$ case after centering in this regime. The variance in the central limit theorems do not change, and the result holds all the way up to the critical temperature $\gb_c$.
\begin{thm}\label{thm-subcr}
	Let $\gb<1$ and $Nh^4 \to 0$. We have
	\begin{align}\label{eq:subc}
		\log Z_N(\gb,h)-\frac12{Nh^2} - \log Z_N(\gb,0) \to 0
	\end{align}
	in probability.
\end{thm}

\begin{proof}
	Note that in the \subc~regime $Nh^4 \to 0$ implies that $(\cosh h)^N = e^{Nh^2/2 +o(1)}$, which can be seen by the Taylor expansion of $\cosh x$. Thus to prove~\eqref{eq:subc}, it is enough to prove that
	\begin{align*}
		\cE_N=\frac{Z_N(\gb,h)}{Z_N(\gb,0)\cosh(h)^N} \to 1\text{ in } L^1.
	\end{align*}
	We define $D_{N}:=\cosh(h)^{-N}Z_N(\gb,h) - Z_N(\gb,0)$ and
	\begin{align*}
		a_N:=\E Z_N(\gb,0) = 2^N\exp\left({(N-1)\gb^2}/4\right).
	\end{align*}
	Note that, here we use the fact that $g_{ij}$'s are i.i.d.~$\N(0,1)$. Using Cauchy-Schwarz inequality, we have
	\begin{align}
		(\E\abs{\cE_N -1})^2 = \left( \E \abs{D_{N}}/Z_N(\gb,0) \right)^2
		 & \le \E ({a_N}^2/Z_N(\gb,0)^2) \cdot a_N^{-2}\E D_{N}^2.\label{eq:sub-err}
	\end{align}
	Now,
	\begin{align}\label{eq:DN}
		\begin{split}
			\E D_{N}^{2}
			&= (\cosh h)^{-2N}\E Z_N(\gb,h)^2 + \E Z_N(\gb,0)^2\\
			&\qquad\qquad- 2(\cosh h)^{-N}\cdot \E Z_N(\gb,0)Z_N(\gb,h).
		\end{split}
	\end{align}
	Recall the definition of $Z_N(\gb,h)$. We compute the following terms:
	\begin{align*}
		\E Z_N(\gb,0)^2,\ \E Z_N(\gb,h)^2 \text{ and } \E Z_N(\gb,0)Z_N(\gb,h).
	\end{align*}
	We take $\mvgs,\mvgt$ as random vectors with each coordinate being i.i.d.~$\pm1$--valued random variable with mean $\hh$. Define $\theta :={\gb}/{\sqrt{N}}$. Thus, we get
	\begin{align*}
		(\cosh h)^{-2N}\E Z_N(\gb,h)^2
		 & = 2^{2N}\E_{\mvgs, \mvgt}\E\exp\bigl(\theta \sum_{i<j} g_{ij}(\gs_i\gs_j +\gt_i\gt_j) \bigr) \\
		 & = 2^{2N}\exp((N-2)\gb^2/2) \E_{\mvgs,\mvgt} \exp\biggl(\frac12{\theta^2}\bigl(\sum_{i=1}^N\gs_i\gt_i \bigr)^2 \biggr).
	\end{align*}
	Using the Gaussian trick, we can write
	\begin{align}\label{eq:subcr-direct-key}
		(\cosh h)^{-2N}\E Z_N(\gb,h)^2
		 & = a_N^2 \exp(-\gb^2/2) \E_{\eta} \bigl(\E_{\gs,\gt}\exp\bigl(\theta\eta \gs\gt\bigr)\bigr)^N
	\end{align}
	where $\eta\sim\N(0,1)$ and independent of everything else. Next, we have
	\begin{align*}
		\E_{\gs,\gt} \exp(x \gs\gt) & =
		\cosh x \cdot \E_{\gs,\gt} (1+\gs\gt\tanh x)
		=\cosh(x)(1+ \hh^2\tanh(x)).
	\end{align*}
	Thus
	\begin{align*}
		(\cosh h)^{-2N}\E Z_N(\gb,h)^2
		 & = a_N^2 \exp(-\gb^2/2) \E(\cosh(\theta \eta))^N(1+\hh^2\tanh(\theta\eta))^N.
	\end{align*}
	Similarly, we have
	\begin{align*}
		(\cosh h)^{-N}\E Z_N(\gb,h)Z_N(\gb,0) = a_N^2 \exp(-\gb^2/2) \E(\cosh(\theta\eta))^N.
	\end{align*}
	Then collecting all terms back into~\eqref{eq:DN}, we get
	\begin{align*}
		a_{N}^{-2} \E\abs{D_N}^2
		 & = \exp(-\gb^2/2)\cdot
		\E\bigl( (\cosh\theta \eta)^N\cdot ((1+\hh^2\tanh(\theta \eta))^N -1)\bigr).
	\end{align*}
	Note that for $\gb<1$, the quantity $\E(\cosh(\theta \eta))^N \to (1-\gb^2)^{-1/2}$ as $N \to \infty$. Moreover, in the sub-critical regime $Nh^4 \to 0$ and thus
	\begin{align*}
		(1+\gh^2\tanh(\theta \eta))^N -1
		= N \hh^2\tanh(\theta \eta) + \frac{N(N-1)}2 \hh^4\tanh^2(\theta \eta)\approx 0.
	\end{align*}
	By dominated convergence theorem, we get that
	\begin{align*}
		\E\bigl( (\cosh\theta \eta)^N\cdot ((1+\hh^2\tanh(\theta \eta))^N -1)\bigr)\to 0 \text{ as } N\to\infty.
	\end{align*}
	Using the argument in the proof of~\cite{Tal03}*{Theorem 2.2.7}, one can easily obtain the boundedness of $\E(a_N^2/Z_N(\gb,0)^2)$. This completes the proof.
\end{proof}

\begin{proof}[Proof of Theorem~\ref{thm:main} Part~1]
	By the classical CLT for $h=0$ proved in~\cites{ALR87} combining with Theorem~\ref{thm-subcr}, the proof of Theorem~\ref{thm:main} in the \subc~regime is immediate.
\end{proof}

\begin{rem}[Generality of the Gaussian trick]
	We believe that this simple approach can be extended to more general spin glass models, like the MSK model and pure even $p$-spin models. Since there is a key step based on the Hubbard-Stratonovich transformation to rewrite $\exp\left( \gb^2/(2N)\cdot (\sum_{i=1}^N \gs_i\gt_i)^2\right)$ as $\E \exp\left( \gb \eta/\sqrt{N}\cdot \sum_{i}\gs_i\gt_{i}\right)$ in~\eqref{eq:subcr-direct-key}, in the MSK model, we will have a quadratic form of overlap, and in the pure even $p$-spin model, the overlap squared will be replaced by a even power of overlaps. The Gaussian trick can still be applied in those settings.
\end{rem} 

\section{Characteristic Approach for Gaussian Disorder}\label{sec:char}
In this section, we establish the framework for the characteristic approach to prove the central limit theorems in different regimes w.r.t.~the external field $h$. We will use $g_{ij}$ instead of $J_{ij}$ to emphasize the $\N(0,1)$ disorder. One of the key tools is the smart path interpolation method due to Guerra.

\subsection{Smart Path Interpolation}\label{ssec:SPI}
For $t\in[0,1]$, we define the interpolating Hamiltonian
\begin{align}\label{eq:interpo}
	H_{N,t}(\mvgs):= \sqrt{t} \cdot\frac{\gb}{\sqrt{N}}\sum_{1\le i<j\le N} g_{ij}\gs_{i}\gs_{j} + \sqrt{1-t} \cdot \gb\sqrt{q}\sum_{i=1}^{N}g_{i}\gs_{i}+ h\sum_{i=1}^{N}\gs_{i}
\end{align}
where $g_{i}$'s are i.i.d.~$\N(0,1)$ random variables independent of the Gaussian disorder $g_{ij}$'s. Define the corresponding partition function
\begin{align*}
	Z_{N,t}=Z_{N,t}(\gb,h)= \sum_{\mvgs\in\gS_{N}} \exp(H_{N,t}(\mvgs)).
\end{align*}
Note that, $Z_{N,1}=Z_{N}$ and
\begin{align*}
	Z_{N,0}=\sum_{\mvgs\in\gS_{N}} \prod_{i=1}^{N} \exp((\gb\sqrt{q}g_{i}+h)\gs_{i})
	=2^{N}\prod_{i=1}^{N}\cosh(\gb\sqrt{q}g_{i}+h)
\end{align*}
satisfies $\var(\log Z_{N,0}) = N\var(\log\cosh(\gb\sqrt{q}g+h))$. We will write $\la\cdot\ra_{t}$ to denote expectation w.r.t.~the Gibbs measure $Z_{N,t}^{-1}\exp(H_{N,t}(\mvgs))$.

The next lemma gives asymptotic behavior for the variance. The proof is given in Section~\ref{pf:auxlem}.
\begin{lem}\label{lem:2}
	We have $\var(\log\cosh(\gb\sqrt{q}g+h)) \le \gb^2q^2$. Moreover,
	\begin{align*}
		\var(\log\cosh(\gb\sqrt{q}g+h)) = \frac12\gb^2 (2-\gb^2)q^2 + O(\gb^2q^3)
	\end{align*}
	for $h\ll 1$.
\end{lem} 

\subsection{Central Limit Theorem}\label{ssec:g-clt}
We fix a sequence $\eta_{N}\ge 1$ and a function $c:[0,1]\to\dR$ that will be specified later. We will write $c_{t}$ for $c(t)$. Define
\begin{align}\label{eq:interpo-free}
	X_{t} :=X_{N,t}(\gb,h)
	 & = \eta_{N}^{-1}\biggl( \log Z_{N,t}- N\E \log 2\cosh(\gb\sqrt{q}g+h)\notag \\
	 & \qquad\qquad\qquad - \frac{Nt}2\gb^2(1-q)^2 - \int_{0}^{t}c_sds\biggr).
\end{align}
For any twice continuously differentiable function $g$ we have
\begin{align}\label{eq:deriv}
	\frac{d}{dt}\E g(X_{t})
	 & = - \frac{\gb^2}4 \eta_{N}^{-2}(1+Nq^2 -c_t) \E g''(X_{t}) \notag \\
	 & \qquad - \frac{\gb^2}{4\eta_{N}}\E \biggl(\bigl(g'(X_{t}) - \eta_{N}^{-1}g''(X_{t})\bigr)\bigl(\la N(R_{12}-q)^2\ra_{t} - c_t\bigr)\biggr).
\end{align}

We define
\begin{align*}
	m(t,x):=\E e^{\text{\i} x X_{t}},\quad x\in\dR
\end{align*}
where $\text{\i}=\sqrt{-1}$ is the complex root of $-1$. Assuming
\begin{align*}
	\eta_N^{-1}\abs{\la N(R_{12}-q)^2\ra_{t} - c_t}\to 0 \text{ in } L^1.
\end{align*}
by taking $g(y)=e^{\text{\i} x y}$, we get from~\eqref{eq:deriv}
\begin{align}\label{eq:dbd}
	 & \abs{\frac{\partial}{\partial t}m(t,x) - \frac{\gb^2}4\cdot \eta_{N}^{-2}(1+Nq^2 -c_t)\cdot x^2m(t,x)}\notag \\
	 & \qquad\qquad\qquad \le
	\frac{\gb^2(|x|+x^2)}{4\eta_{N}}\E \bigl|\la N(R_{12}-q)^2\ra_{t} - c_t\bigr|.
\end{align}
Define, the function
\begin{align}\label{eq:vt}
	V(t):= \frac12\gb^2\int_0^t \eta_{N}^{-2}(1+Nq^2 -c_s)\, ds,\qquad t\in[0,1].
\end{align}
Multiplying both sides of~\eqref{eq:dbd} by $\exp(-V(t)x^2/2)$ and integrating, we get
\begin{align}\label{eq:charerr}
	\begin{split}
		&\abs{\exp(-V(t)x^2/2)m(t,x) - m(0,x)}\\
		&\qquad\qquad\le
		\frac{\gb^2(|x|+x^2)}{4\eta_{N}}
		\int_0^t\E \bigl|\la N(R_{12}-q)^2\ra_{s} - c_s\bigr|\, ds.
	\end{split}
\end{align}
Finally, we note that, $m(0,x)\to \exp(-x^2a/2)$ where
\begin{align*}
	a=\lim \eta_{N}^{-2}N\var(\cosh(\gb\sqrt{q}g+h))
	=\begin{cases}
		0 & \text{ if } \ga > 1/4 \\
		\frac12\gb^2 (2-\gb^2)\cdot \frac{\rho^4}{(1-\gb^2)^2} & \text{ if } \ga = 1/4 \\
		\frac12\gb^2 (2-\gb^2)\cdot \frac{1}{(1-\gb^2)^2} & \text{ if } \ga < 1/4
	\end{cases}
\end{align*}
when we take $\eta_N^2 = 1$ for $\ga \ge 1/4$ and $\eta_N^2 = Nh^4$ for $\ga < 1/4$. 

Thus to prove central limit theorems, notice that in the expression~\eqref{eq:charerr}, we need to show that
\begin{align}\label{eq:l1-conver}
	\frac{1}{\eta_{N}}\cdot \int_0^t\E \bigl|\la N(R_{12}-q)^2\ra_{s} - c_s\bigr|\, ds \to 0 \text{ as } N \to \infty.
\end{align}
In the \supc~regime, since $\eta_N \gg 1$, with $c\equiv0$, we only need that $N\la (R_{1,2}-q)^2\ra_t$
is uniformly bounded in $t$. This can be achieved by adapting Latala's argument at very high temperature, see~\cite{Tal11a}*{Section 1.4}.

For the \subc~and \crit~regimes, we will take $\eta_N=1$ and $c_t=1/(1-\gb^2 t), t\in [0,1]$ motivated by the behaviour at $h=0$ case. Note that
\begin{align*}
	\int_0^1c_s\,ds = -\gb^{-2}\log(1-\gb^2).
\end{align*}

Combining everything, we get that
\begin{align*}
	m(1,x)\to \exp(-x^2 v^2/2) \text{ for all } x\in\dR
\end{align*}
where $v^2=a-\lim V(1)$. Simplifying, we get
\begin{align*}
	v^2
	= \frac12\gb^2 (2-\gb^2)\cdot \frac{1}{(1-\gb^2)^2} - \frac{\gb^2}{2} \cdot \frac{1}{(1-\gb^2)^2}
	= \frac{\gb^2}{2(1-\gb^2)}
\end{align*}
when $\ga<1/4$ and
\begin{align*}
	v^2 & := \frac12\gb^2 (2-\gb^2)\cdot \frac{\rho^4}{(1-\gb^2)^2} - \frac{\gb^2}{2}\left(1+ \frac{\rho^4}{(1-\gb^2)^2}\right) -\frac12\log(1-\gb^2) \\
	 & = \frac{\gb^2\rho^4}{2(1-\gb^2)} - \frac{\gb^2}{2} -\frac12\log(1-\gb^2)
\end{align*}
when $\ga\ge 1/4$. Here we used the fact that $q^2\approx h^4/(1-\gb^2)^2$ from Lemma~\ref{lem:1}.  

Going back to proving~\eqref{eq:l1-conver}, we will first prove a uniform bound on $N\E \la (R_{12} -q)^2 \ra_t$ and then use Dominated Convergence Theorem by showing that, for all $t\in (0,1]$, we have
\begin{align}\label{eq:conv0}
	\E \abs{ \la N (R_{12}-q)^2 \ra_t -c_t} \to 0 \text{ as } N\to\infty
\end{align}
In the positive external field case, there is a beautiful argument originally due to Latala, which gives the exponential tightness of the overlap at very high temperature. It turns out that for weak external field, this argument can be easily adapted.

\begin{thm}[\cite{Tal11a}*{Theorem 1.4.1}]\label{thm:latala}
	Assume that $s+2\gb^2<1/2$, for all $0\le t \le 1$ we have
	\begin{align*}
		\E \la \exp(sN(R_{12}-q)^2)\ra_{t} \le \frac{1}{\sqrt{1-2s-4t\gb^2}},
	\end{align*}
	where $q$ is the unique solution to $q= \E \tanh^2(\gb \eta \sqrt{q}+h)$.
\end{thm}

The proof is based on the smart path interpolation method, where the overlap in the decoupled model ($t=0$) can be controlled easily. The desired result can be obtained by the monotone property along the interpolation path. In the proof, the external field term is essentially auxiliary. Therefore, the same proof will go through for the weak external field case. 

By Cauchy-Schwarz inequality, we have
\begin{align}\label{eq:cs-ineq}
	\E \abs{ \la N (R_{12}-q)^2 \ra_t -c_t} & = \E \left( U_{N,t} \cdot \frac{(\E Z_{N,t}(\gb,h))^2}{Z_{N,t}(\gb,h)^2} \right)\notag \\
	 & \le \left( \E U_{N,t}^2 \right)^{1/2} \left( \E\left( \frac{(\E Z_{N,t}(\gb,h))^4}{Z_{N,t}(\gb,h)^4}\right) \right)^{1/2}
\end{align}
where
\begin{align*}
	U_{N,t}:= \frac{\sum_{\mvgs,\mvgt} \exp\left(H_{N,t}(\mvgs)+H_{N,t}(\mvgt)\right)\cdot (N(R_{12}-q)^2 -c_t)}{(\E Z_{N,t}(\gb,h))^2}.
\end{align*}
In the following subsections, we develop an extended quadratic coupling method to argue that
\begin{align}\label{eq:subcrit-res}
	\E U_{N,t}^2 \to 0 \ \text{ as} \ N \to \infty
	\text{ and }
	\E\left( |\E Z_{N,t}(\gb,h)|^4/{Z_{N,t}(\gb,h)^4} \right) < \infty.
\end{align}

\subsection{Quadratic Coupling Method to Control Overlap}
The quadratic coupling method was originally used to study the SK model without external field by Guerra and Toninelli~\cite{GT02}. This Section extends it to the case when an external field is present in the finite volume system. To be precise, we show that the Gibbs average of overlap is exponentially tight when $Nh^4 \approx O(1)$. This result could be of independent interest and used to study other questions.

We introduce the following lemma to control the overlap.
\begin{lem}\label{lem:quad-coup}
	Assume that $\gc +\gb^2 <1$ and $\ga\ge 1/4$. We have
	\begin{align*}
		\E \sum_{\mvgs,\mvgt} \exp \biggl( H_N(\mvgs) + H_N(\mvgt) & + h \sum_{i=1}^N (\gs_i + \gt_i) + \frac12{\gc N }(R_{1,2}-q)^2\biggr) \\
		 & \qquad \le K\cdot (\E Z_N(\gb,h))^2,
	\end{align*}
	where $K$ is a finite constant depending on $\gb,\rho,\ga,\gc$. Moreover, same result holds if we replace $q$ by any $r$ such that $Nr^2$ is bounded by a constant.
\end{lem}
\begin{proof}[Proof of Lemma~\ref{lem:quad-coup}]
	Notice the following basic facts,
	\begin{align*}
		\E Z_N(\gb,h)
		 & = \sum_{\mvgs} \exp\bigl({(N-1)\gb^2}/4 + h \sum_i \gs_i \bigr)
		= (2\cosh h)^N \exp((N-1)\gb^2/4) \\
		\text{ and }
		\E Z_N(\gb,h)^2
		 & = \sum_{\mvgs,\mvgt} \exp\bigl( (N-2)\gb^2/4+ \gb^2 R_{1,2}^2/4 + h \sum_i (\gs_i+ \gt_i) \bigr).
	\end{align*}
	Thus, we have,
	\begin{align*}
		A_{\gc} & :=\E \bigl( Z_N(\gb,h)^2 \bigl\la \exp\bigl({\gc N} (R_{1,2}-q)^2/2\bigr)\bigr\ra \bigr)	/ \bigl(\E Z_N(\gb,h)\bigr)^2 \\
		 & \;= \E_{\mvgs,\mvgt} \exp \left( \gb^2(NR_{1,2}^2-1)/2 + {\gc N}(R_{1,2}-q)^2/2 \right).
	\end{align*}
	Simplifying and using the Gaussian trick we get
	\begin{align*}
		A_{\gc}= e^{(\gc N q^2 -\gb^2)/2} \cdot \E_{\eta,\mvgs,\mvgt} \exp \left( \sum_{i=1}^N \left[\eta\cdot \sqrt{(\gb^2 +\gc)/{N}} - q \gc\right] \gs_i\gt_i\right)
	\end{align*}
	where $\eta\sim\N(0,1)$. Next, we denote $\xi:= \eta\cdot\sqrt{\gb^2+\gc} - \gc \sqrt{Nq^2}$, to simplify the expression as
	\begin{align*}
		A_{\gc}
		 & = e^{(\gc N q^2 -\gb^2)/2} \cdot \E_{\xi} \left(\cosh (\xi/\sqrt{N}))^N \left(1+ \hh^2 \tanh(\xi/\sqrt{N})\right)^N\right).
	\end{align*}
	The last step can be further upper bounded by
	\begin{align}\label{eq:ub1}
		A_{\gc} & = e^{(\gc N q^2 -\gb^2)/2}\cdot \E_{\xi} \exp\left({\xi^2}/2 + N\hh^2\tanh\left({\xi}/{\sqrt{N}}\right)\right) \notag \\
		 & \le e^{(\gc N q^2 -\gb^2)/2} \cdot\E_{\xi} \exp\left({\xi^2}/2 + \sqrt{Nh^4}\cdot \abs{\xi}\right).
	\end{align}
	Now we use the fact that $\xi:= \eta\cdot\sqrt{\gb^2+\gc} - \gc \sqrt{Nq^2}$ to simplify the right hand side expression in~\eqref{eq:ub1}, to get the upper bound
	\begin{align}\label{eq:ub2}
		 & A_{\gc}\le e^{(\gc N q^2 -\gb^2)/2} \cdot\theta^{-1} \\
		 & \cdot
		\E_{\eta}
		\exp\left( \gc^2 Nq^2 - \eta/(2\theta)\cdot \sqrt{\gc^2Nq^2} \cdot\sqrt{\gb^2+\gc}
		+
		\sqrt{Nh^4}\cdot \abs{\eta/\theta\cdot\sqrt{\gb^2+\gc} - \gc \sqrt{Nq^2}}
		\right) \notag
	\end{align}
	with $\theta^2=1-\gb^2-\gc>0$.
	It is now clear that under the conditions $\gb^2 + \gc <1$ and the fact that $Nh^4 \approx Nq^2(1-\gb^2)^{-1}$ is bounded by a constant, the right hand side of~\eqref{eq:ub2} is bounded by a finite constant depending on $\gb,\rho, \ga,\gc$.
\end{proof}   

\subsection{Tail Control of the log-partition function}\label{ssec:tail-logp}
In this section, we derive the following tail behavior of the log-partition function. We essentially follow the same steps as in the $h=0$ case in~\cite{Tal11b}*{Section~11.2}.

\begin{thm}\label{thm:tail-logp}
	Let $\gb<1$, and $\ga\ge 1/4$. There exists a constant $K=K(\gb,\rho,\ga)$ such that for all large $N$, and all $t>0$, we have
	\begin{align*}
		\pr \left( \log Z_N(\gb,h) \le N \left({\gb^2}/4 + \log (2\cosh h)\right) -t \right) \le K\cdot \exp\left(-t^2/K\right).
	\end{align*}
\end{thm}

The above tail control of the log-partition function is crucial for the moment bounds on the partition function, which will be used to prove the central limit theorems in the \crit~and \subc~regimes using Gaussian interpolation argument. We formulate it as the following corollary.

\begin{cor}\label{cor:moment-part}
	Let $\gb<1$, and $\ga\ge 1/4$. For any $p>0$, there exists a constant $m_p<\infty$, such that
	\begin{align*}
		\E \left( \left({\E Z_N(\gb,h)}/{Z_N(\gb,h)}\right)^p\right) \le m_p.
	\end{align*}
\end{cor}
\begin{proof}
	The proof is immediate by Theorem~\ref{thm:tail-logp} and standard moment inequality.
\end{proof}

To prove Theorem~\ref{thm:tail-logp}, we need the following Lemma.
\begin{lem}\label{lem:suit-k}
	Let $\gb<1$, and $\ga\ge 1/4$. We have
	\begin{align*}
		\pr\left( Z_N(\gb,h) \ge \frac12\E Z_N(\gb,h), N \la R_{1,2}^2\ra \le K \right) \ge \frac 1 K.
	\end{align*}
\end{lem}

\begin{proof}
	Take $r=0$ in the extended quadratic coupling bound introduced in Lemma~\ref{lem:quad-coup}, we have
	\begin{align*}
		\E \left( Z_N(\gb,h)^2 \la \exp(\gc N R_{1,2}^2/2)\ra_{\gb,h} \right) \le C\cdot (\E Z_N(\gb,h))^2
	\end{align*}
	for some finite constant $C>0$.
	Using Markov's inequality,
	\begin{align*}
		\pr \biggl( Z_N(\gb,h)^2 \la \exp(\gc N R_{12}^2/2) \ra_{\gb,h} \le t (\E Z_N(\gb,h))^2 \biggr) \ge 1 -C/t.
	\end{align*}
	By Paley--Zygmund inequality, we have
	\begin{align*}
		\pr \left( Z_N(\gb,h) \ge \frac12\E Z_N(\gb,h)\right) & \ge \frac 1 4 \frac{(\E Z_N(\gb,h))^2}{\E Z_N(\gb,h)^2} \ge C',
	\end{align*}
	where the second inequality is due to Lemma~\ref{lem:quad-coup} with $r=\gc=0$. Using the elementary inequality $\pr(A \cap B) \ge \pr(A) + \pr(B) -1$, by taking $t=K=(1+C)/C'$, we have
	\begin{align*}
		\pr \left( Z_N(\gb,h)^2 \la \exp(\gc N R_{12}^2/2) \ra \le K (\E Z_N(\gb,h))^2,\ Z_N(\gb,h) \ge \frac12{\E Z_N(\gb,h)}\right) \ge \frac1K.
	\end{align*}
	Using $Z_N(\gb,h) \ge \frac12\E Z_N(\gb,h)$, and combining the fact $e^x \ge x$, finally we have
	\begin{align*}
		Z_N(\gb,h)^2 \la \exp(\gc N R_{12}^2 /2 ) \ra \le K (\E Z_N(\gb,h))^2 \ \text{ implies} \ N \la R_{12}^2 \ra \le K.
	\end{align*}
	This completes the proof.
\end{proof} 

\begin{lem}
	[{\cite{Tal03}*{Lemma~2.2.11}}]\label{lemm-talagrand} Consider a closed subset $B$ of $\dR^M$, and set
	\begin{align*} d(\mvx,B) = \inf\{d(\mvx,\mvy);\mvy \in B\}, \end{align*}
	as the Euclidean distance from $\mvx$ to $B$. Then for $t>0$, we have
	\begin{align}\label{eq0-h0}
		\pr\left(d(\mveta,B) \ge t+ 2 \sqrt{\log(2/\pr(\mveta\in B)}) \right) \le 2 \exp(-t^2/4)
	\end{align}
	where $\mveta = (\eta_1,\eta_2,\ldots, \eta_m)$ and $\eta_i \sim \N(0,1)$ are i.i.d.
\end{lem}

Now we prove Theorem~\ref{thm:tail-logp}.
\begin{proof}[Proof of Theorem~\ref{thm:tail-logp}]

	Recall that the Hamiltonian is
	\begin{align*}
		H_N(\mvgs) = \frac{\gb}{\sqrt{N}} \sum_{i<j} g_{ij}\gs_i\gs_j + h \sum_{i=1}\gs_i.
	\end{align*}
	Let $M=N(N-1)/2$, and consider Gaussian $\mveta = (\eta_{ij})_{i<j}$, in this case, $\mveta \in \dR^M$. We understand $H_N(\mvgs), Z_N(\gb,h)$ as functions of $\mveta$. For convenience, we drop the dependence of $Z_N(\gb,h)$ on $\gb,h$ temporarily in the following proof. By Lemma~\ref{lem:suit-k}, for some suitably large $K=K_1$, there is a subset $B \subset \dR^M$ with
	\begin{align*} \{\mveta \in B \} = \{ Z_N(\mveta) \ge \E Z_N/2; N\la R_{12}^2\ra \le K_1 \} \end{align*}
	i.e. the set $B$ characterizes the event on RHS, and also $\pr(\mveta \in B) \ge \frac 1 {K_1}$. Next we will prove
	\begin{align}\label{eq1-h0}
		\log Z_N(\mveta) \ge N\left(\log 2 + {\gb^2}/4 + \log \cosh(h)\right)- K_1(1+d(\mveta,B))
	\end{align}
	For $\mveta' \in B$, we have
	\begin{align*}
		\log Z_N(\mveta') \ge \log \frac 1 2 \E Z_N = (N-1) \left({\gb^2}/4 + \log 2 \right) + N \log \cosh(h).
	\end{align*} To prove~\eqref{eq1-h0}, it is enough to show that
	\begin{align}\label{eq2-h0}
		\log Z_N(\mveta) \ge \log Z_N(\mveta') - K_1 d(\mveta,\mveta')
	\end{align}
	for all $\mveta'\in B$. Here $K_{1}$ must be bigger than $\log 2 +{\gb^2}/4$. Notice that
	\begin{align*}
		Z_N(\mveta) = Z_N(\mveta') \left\la \exp \biggl( \frac{\gb}{\sqrt{N}} \sum_{i<j}(\mveta_{ij}-\mveta_{ij}')\gs_i\gs_j \biggr) \right \ra'
	\end{align*}
	where $\la \cdot \ra'$ denotes the Gibbs average with disorders $\mveta'$. Since there is an exponential part, it is natural to apply Jensen's inequality,
	\begin{align*}
		\left\la \exp \biggl( \frac{\gb}{\sqrt{N}} \sum_{i<j}(\mveta_{ij}-\mveta_{ij}')\gs_i\gs_j \biggr) \right \ra' \ge \exp \biggl( \frac{\gb}{\sqrt{N}} \sum_{i<j}(\mveta_{ij}-\mveta_{ij}')\la \gs_i\gs_j \ra' \biggr).
	\end{align*}
	Applying Cauchy-Schwarz inequality, we have
	\begin{align*}
		\sum_{i<j}(\mveta_{ij}-\mveta_{ij}')\la \gs_i\gs_j \ra' \ge -d(\mveta,\mveta')(N^2\la R_{12}^2 \ra')^{1/2} \ge -K_1\sqrt{N}d(\mveta,\mveta')
	\end{align*}
	where $d(\mveta,\mveta') = \sum_{i<j}(\eta_{ij}-\eta_{ij}')^2$. The last inequality is based on $\mveta' \in B$, then this proves~\eqref{eq2-h0}, and hence~\eqref{eq1-h0}. From~\eqref{eq1-h0}, we can easily see that
	\begin{align*}
		 & \log Z_N(\mveta) \le N\left({\gb^2}/4 + \log 2 + \log \cosh(h) \right) -t \\
		 & \qquad\qquad \Longrightarrow d(\mveta,B) \ge \frac{t-K_1}{K_1} \ge 2 \sqrt{\log(2/\pr(\mveta \in B))} +\frac{t-K_2}{K_1}.
	\end{align*}
	By~\eqref{eq0-h0} in Lemma~\ref{lemm-talagrand}, it follows that
	\begin{align*}
		\pr \left( \log Z_N(\mveta) \le N \left({\gb^2}/4 +\log 2 + \log \cosh(h) \right)-t\right) \le 2 \exp \left( -(t-K_2)^2/(4K_1^2) \right)
	\end{align*}
	for $t \ge K_2$. Now for the RHS, if we take $K$ large enough, when $t\ge K_2$,
	\begin{align*}
		2 \exp \left( -{(t-K_2)^2}/{4K_1^2} \right) \le K \exp \left(-t^2/{K} \right).
	\end{align*}
	On the other hand, when $0 \le t \le K_2$, we have $ K \exp \left(-{t^2}/{K} \right) \ge 1$. Therefore, in any case, we proved
	\begin{align*}
		\pr \left( \log Z_N(\mveta) \le N \left( {\gb^2}/4 +\log 2 + \log \cosh(h) \right)-t\right) \le K \exp \left(-t^2/{K} \right).
	\end{align*}
	Note that, in this proof we used $K$ to denote a generic constant depending only on $\gb,\ga,\rho$. This completes the proof.
\end{proof}   

\subsection{Second Moment Control of $U_{N,t}$}\label{sec:new-pf}
In this section, we are devoted to prove that $\E U_{N,t}^2 \to 0 $ as $N \to \infty$. Recall from Section~\ref{ssec:SPI} that,
\begin{align*}
	Z_{N,t}(\gb,h) := \sum_{\mvgs} \exp\left(\frac{\gb\sqrt{t}}{\sqrt{N}} \sum_{i<j} g_{ij} \gs_i\gs_j + \sum_{i=1}^N \left(\sqrt{1-t} \cdot \gb g_{i} \sqrt{q} + h\right)\gs_i\right).
\end{align*}

\begin{thm}\label{thm:2nd-mom}
	There exists $\gb_{0}< 1$ such that when $\gb < \gb_0, \ga\ge 1/4$ and $t>0$, we have
	\begin{align*}
		\E U_{N,t}^2 \to 0 \text{ as } N \to \infty.
	\end{align*}
\end{thm}
\begin{proof}[Proof of Theorem~\ref{thm:2nd-mom}]

	We carry out the computations for $\E U_{N,t}^2$ as follows. For convenience, we use the following notations
	\begin{align*}
		\wmvgs & := (\mvgs^1,\mvgs^2,\mvgs^3,\mvgs^4),
	\end{align*}
	where for each $i$, $\mvgs^i:=(\gs_1^i,\gs_2^i,\ldots,\gs_N^i)$ is a vector of $N$ many i.i.d.~$\pm 1$ valued random variables with mean $\hh$. We also let
	\begin{align*}
		H_{N,t}(\wmvgs)
		:= \sum_{k=1}^4 H_{N,t}(\mvgs^k)
		\text{ and }
		\bar{R}_{12} := R_{12}-q.
	\end{align*}
	Define
	\begin{align*}
		a_{N,t}:=\E Z_{N,t}(\gb,h) = (2\cosh h)^N \exp\bigl((N-1)\gb^2t/4 +\gb^{2}(1-t)qN/2\bigr).
	\end{align*}
	Taking expectation w.r.t.~the Gaussian disorder, we have
	\begin{align}\label{eq:s2}
		\begin{split}
			\E U_{N,t}^2
			& = a_{N}^{-4}\sum_{\wmvgs} \bigl( N(R_{12} -q)^2 - c_t\bigr)\cdot\bigl( N(R_{34} -q)^2 - c_t\bigr)\cdot \E\exp \left( H_{N,t}(\wmvgs) \right)\\
			& = \exp(-3\gb^{2}t) \cdot \E_{\wmvgs} \biggl((N\bar{R}_{12}^2 - c_t) (N\bar{R}_{34}^2 - c_t)\\
			&\qquad\qquad\qquad\qquad\qquad \cdot \exp \bigl( \frac12\sum_{1\le k<l\le 4} {\gb^2t}N R_{kl}^2 + 2(1-t) \gb^2 qN R_{kl}\bigr) \biggr).
		\end{split}
	\end{align}
	Using $\bar{R}_{kl}=R_{kl}-q$ and simplifying we get from~\eqref{eq:s2}
	\begin{align}\label{eq:s3}
		\begin{split}
			\E U_{N,t}^2 &= \exp( -3\gb^{2}t+\gb^2(2-t) Nq^2/2) \\
			&\quad\cdot \E_{\wmvgs} \biggl((N\bar{R}_{12}^2 - c_t) (N\bar{R}_{34}^2 - c_t)\cdot \exp \bigl( \frac12\sum_{1\le k<l\le 4} {\gb^2t}N \bar{R}_{kl}^2 + 2\gb^2 qN\bar{R}_{kl}\bigr) \biggr).
		\end{split}
	\end{align}
	Note that, $K:=\gb^{-4}\exp( -3\gb^{2}t+{\gb^2(2-t)} Nq^2/2)$ is converging to a finite positive constant. We can use second order Gaussian integration by parts
	$$
		a^2e^{a^{2}/2} = a^2\E e^{a\eta}=\E(\eta^{2}-1)e^{a\eta}
	$$
	for $\eta\sim\N(0,1)$ and $1+\gb^{2}tc_{t}=c_{t}$, to get
	\begin{align}\label{eq:s4}
		\begin{split}
			\E U_{N,t}^2 &= Kt^{-2}\cdot \E_{\wmvgs, \mveta} \biggl( (\eta_{12}^2 -c_t)(\eta_{34}^2-c_t)\\
			&\qquad\qquad\qquad \cdot \exp\biggl( \sum_{1\le k<l \le 4} \gb \sqrt{t}\cdot \eta_{kl} \sqrt{N} \bar{R}_{kl} + \gb^2 \sqrt{Nq^2} \cdot \sqrt{N}\bar{R}_{kl}\biggr)\biggr),
		\end{split}
	\end{align}
	where $\eta_{kl},k<l$ are i.i.d.~$\N(0,1)$. Now rearranging, we get
	\begin{align*}
		 & \E_{\wmvgs, \mveta} \biggl( (\eta_{12}^2 -c_t)(\eta_{34}^2-c_t)\cdot \exp\biggl( \sum_{1\le k<l \le 4} \gb \sqrt{t}\cdot \eta_{kl} \sqrt{N} \bar{R}_{kl} + \gb^2 \sqrt{Nq^2} \cdot \sqrt{N}\bar{R}_{kl}\biggr)\biggr) \\
		 & = \E_{\mveta} \biggl( (\eta_{12}^2 -c_t)(\eta_{34}^2-c_t) \\
		 & \qquad\qquad\qquad \cdot \E_{\wmvgs}\exp\biggl( \sum_{1\le k<l \le 4} \left(\sqrt{t} \cdot \eta_{kl} + \gb \sqrt{Nq^2}\right)\left(\gb\sqrt{N}R_{kl} -\gb\sqrt{Nq^{2}}\right)\biggr)\biggr) \\
		 & = \E_{\mveta} \biggl( (\eta_{12}^2 -c_t)(\eta_{34}^2-c_t)\prod_{k<l}\exp(\theta(\sqrt{t} \eta_{kl} -\theta)) \\
		 & \qquad\qquad\qquad\cdot \E_{\wmvgs}\exp\biggl({\gb}N^{-1/2}\sum_{k<l}(\sqrt{t} \eta_{kl} -\theta) \sum_{i=1}^N \gs_i^k\gs_i^l \biggr)\biggr)
	\end{align*}
	where $\theta:= -\gb\sqrt{Nq^2}\to \theta_{\infty}:= -\gb\rho^{2}/(1-\gb^{2})$. Define
	\begin{align*}
		\hat\eta_{kl}:=\sqrt{t} \eta_{kl} -\theta, 1\le k<l\le 4.
	\end{align*}
	We notice that
	\begin{align*}
		 & \exp\biggl({\gb}N^{-1/2}\sum_{k<l}\hat\eta_{kl} \sum_{i=1}^N \gs_i^k\gs_i^l \biggr) \\
		 & \qquad		= \prod_{i=1}^N \prod_{k<l}\cosh\left(\gb\hat\eta_{kl}/\sqrt{N}\right) \left( 1+ \gs_i^k \gs_i^l \tanh\left( \gb\hat\eta_{kl}/\sqrt{N} \right)\right)
	\end{align*}
	to write
	\begin{align}\label{eq:unt-sqr}
		 & \E U_{N,t}^2 \\
		 & = Kt^{-2}\E_{\mveta}\biggl( (\eta_{12}^2 -c_t)(\eta_{34}^2-c_t)\prod_{k<l}\exp(\theta \hat\eta_{kl} + N\log \cosh(\gb\hat\eta_{kl}/\sqrt{N})) \cdot\exp(g_{N}(\hat\mveta))\biggr)\notag
	\end{align}
	where
	\begin{align}
		g_{N}(\mvy) & := N\log\E_{\gs} \prod_{k <l} \left( 1+ \gs_{k}\gs_{l} \tanh(\gb y_{kl}/\sqrt{N}) \right).
		\label{eq:int-eq2}
	\end{align}

	Using the fact that $$\sqrt{N}\E\gs_{k}\gs_{l}=\sqrt{N}\hh^2\approx \sqrt{Nh^4}\to \rho^{2} =-(1-\gb^{2})\theta/\gb,$$ we note that the term inside $\E_{\mveta}$ in~\eqref{eq:unt-sqr} is converging pointwise to
	\begin{align*}
		 & (\eta_{12}^2 -c_t)(\eta_{34}^2-c_t)\prod_{k<l}\exp\left((\theta_{\infty}+\gb\rho^{2})(\sqrt{t}\eta_{kl} -\theta_{\infty}) + \frac12\gb^{2}(\sqrt{t}\eta_{kl} -\theta_{\infty})^2\right) \\
		 & = (\eta_{12}^2 -c_t)(\eta_{34}^2-c_t)\prod_{k<l}\exp\left(\frac12\gb^{2}t\eta_{kl}^{2} - \gb^{2}\theta_{\infty}^2/2\right);
	\end{align*}
	in the last equality we used the fact that $(1-\gb^{2})\theta_{\infty}+\gb\rho^{2}=0$.
	Applying Lemma~\ref{lem:key}, $g_{N}$ is upper bounded by
	\begin{align*}
		g_{N}(\mvy)\le (5/2+C'h^2)\gb^2 \sum_{k<l} y_{kl}^2 +\gb\rho^{2} \sum_{k<l}\abs{\hat\eta_{kl}}.
	\end{align*}
	In particular, the expression~\eqref{eq:unt-sqr} is upper bounded by
	\begin{align*}
		\E_{\mveta}\biggl( (\eta_{12}^2 -c_t)(\eta_{34}^2-c_t)\prod_{k<l}\exp\bigl((|\theta| +\gb\rho^{2}) \abs{\hat\eta_{kl}} + (3+C'h^2)\gb^2\hat\eta_{kl}^2\bigr)\biggr).
	\end{align*}
	To guarantee the above integral is finite, we need $
		1/2-(3+C'h^2)\gb^2t > 0$, or equivalently, $\gb^2 t < 1/6$. Thus if $\gb^2< \gb_0^2:=1/6$, by dominated convergence theorem and Lemma~\ref{lem:tanh-inte}, it is clear that the expectation in~\eqref{eq:unt-sqr} tends to 0 as $N \to \infty$.
\end{proof}

\begin{lem}\label{lem:tanh-inte}
	Let $(\eta_{kl})_{1\le k < l \le 4}$ be i.i.d.~$\N(0,1)$. Then
	\begin{align*}
		\E_{\mveta} \biggl((\eta_{12}^2-c_t)(\eta_{34}^2-c_t)\cdot \exp\biggl( \frac12\gb^{2}t\sum_{k <l} \eta_{kl}^2\biggr) \biggr) = 0.
	\end{align*}
\end{lem}
\begin{proof}[Proof of Lemma~\ref{lem:tanh-inte}]
	The proof of this lemma is elementary by integrating with respect to a tilted Gaussian measure,
	\begin{align*}
		 & \frac{1}{(2\pi)^{3}}\int_{\dR^6}(y_{12}^2-c_t)(y_{34}^2-c_t) \exp\biggl( -\frac12(1-\gb^2t)\sum_{k <l} y_{kl}^2\biggr) d\mvy \\
		 & \qquad\qquad = \frac{1}{(2\pi)^{3}}\int_{\dR^6}(y_{12}^2-c_t)(y_{34}^2-c_t) \exp\biggl( -\frac1{2c_t}\sum_{k <l} y_{kl}^2\biggr) d\mvy =0.
	\end{align*}
	The proof is complete.
\end{proof}

The following lemma is used to control the expectation in~\eqref{eq:int-eq2}.
\begin{lem}\label{lem:key}
	For any real numbers $(y_{kl})_{1\le k < l \le 4}$, we have
	\begin{align*}
		 & N\log\E_{\gs} \prod_{1\le k <l \le 4} \left( 1+ \gs_k\gs_l \tanh(y_{kl}/{\sqrt{N}}) \right) \\
		 & \qquad\qquad \le (5/2+C'h^2)\sum_{1\le k<l\le 4} y_{kl}^2+\sqrt{Nh^4}\sum_{1\le k<l\le 4} \abs{y_{kl}},
	\end{align*}
	where $C'$ is some finite constant independent of $N,h$.
\end{lem}
\begin{proof}[Proof of the Lemma~\ref{lem:key}]
	Let $x_{kl}:={y_{kl}}/{\sqrt{N}}$. We first expand the left hand side by constructing a correspondence to a collection of graph structures,
	\begin{align}\label{eq:exp-tanh}
		f_{N}:=\E_{\gs_i} \prod_{1 \le k < l \le 4}( 1 + \gs_k \gs_l \tanh x_{kl})
		= 1+ \sum_{r,c=1}^3 \sum_{\gC \in \cG_{r,c}} \hh^{\abs{\partial \gC}} \prod_{e \in \gC}\tanh x_e,
	\end{align}
	where $\cG_{r,c}$ denotes the collection of graphs that are topologically same as the structure in $r$--th row and $c$--th column in Figure~\ref{fig:graphs}.

	\begin{figure}[htbp]
		\centering
		\includegraphics[scale=0.6]{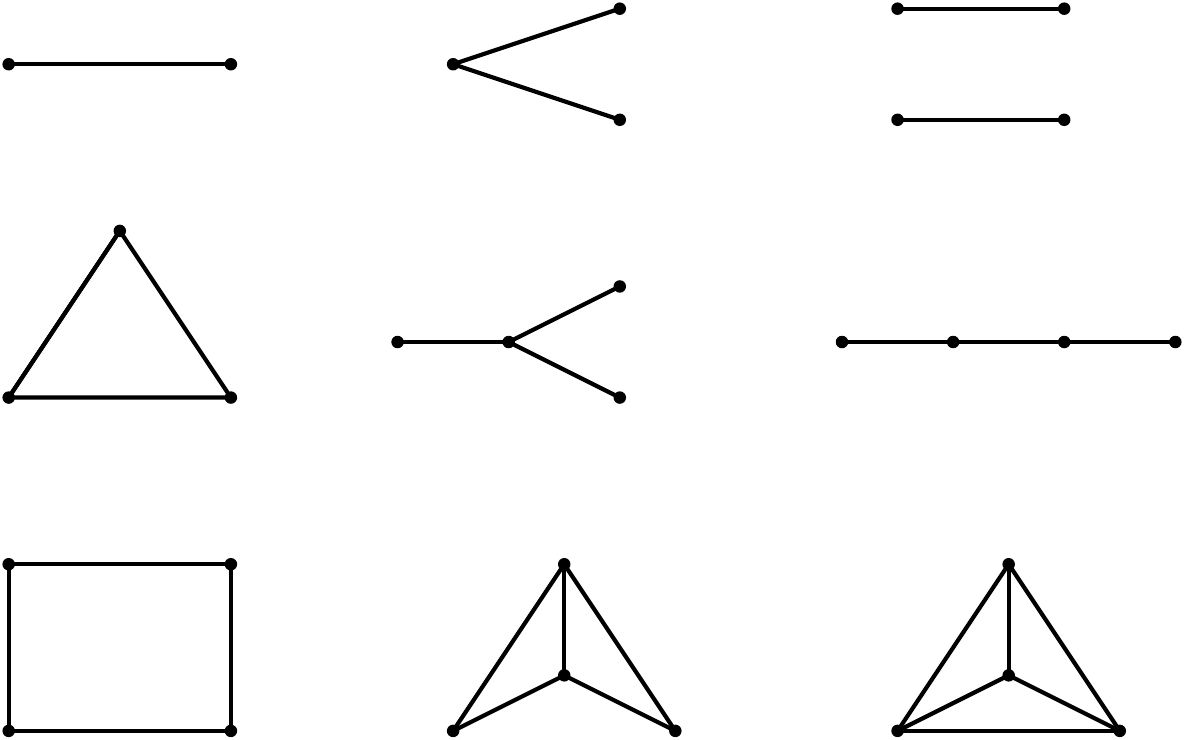}
		\caption{The graph structures correspond to the terms in the expansion of~\eqref{eq:exp-tanh}. There are 4 vertices in total and can have at most 6 edges. Let $\cG_{r,c} \ \text{ for} \ r,c=1,2,3$ denote the collection of graphs with the same topological structure in $r$--th row and $c$--th column. Each edge corresponds to a term $\gs_i^k\gs_i^l \tanh(x_{kl})$. }
		\label{fig:graphs}
	\end{figure}
	For a graph $\gC$ in the collection $\cG_{r,c}$, let $\abs{\partial \gC}$ be the number of odd degree vertices in $\gC$. Let $f_N$ be the right handside of~\eqref{eq:exp-tanh}, one has
	\begin{align}\label{eq:grdecomp}
		\log f_N
		 & = N \log\bigl(1+ \sum_{r,c} \sum_{\gC \in \cG_{r,c}} \hh^{\abs{\partial \gC}} \prod_{e \in \gC}\tanh x_e\bigr)
		\le N \sum_{r,c} \sum_{\gC \in \cG_{r,c}} \hh^{\abs{\partial \gC}} \prod_{e \in \gC}\tanh \abs{x_e}.
	\end{align}
	For the sum over $\cG_{1,1}$, the structure is relatively simple with a single bond and two odd degree vertices, using $\tanh\abs{x}\le x$, we get
	\begin{align}\label{eq:leading}
		N \sum_{\gC \in \cG_{1,1}} \hh^2 \prod_{e \in \gC} \tanh\abs{x_e} \le \sum_{1\le k<l \le 4} \sqrt{N} h^2 \abs{y_{kl}}.
	\end{align}
	Using AM--GM inequality and the fact that $\tanh\abs{x}\le \min(1,\abs{x})$ for all $x$, we get from~\eqref{eq:grdecomp}
	\begin{align*}
		\log f_{N}\le \sqrt{Nh^4}\sum_{1\le k<l \le 4} \abs{y_{kl}} +\sum_{r+c>2} \sum_{\gC \in \cG_{r,c}} \hh^{\abs{\partial \gC}}\cdot \frac1{|\gC|}\sum_{e \in \gC}y_e^2.
	\end{align*}
	Note that, $\abs{\partial \gC}$ is always even and it is zero only for the cycle graphs in $\cG_{2,1},\cG_{3,1}$. Moreover, $\abs{\cG_{2,1}}/3+\abs{\cG_{3,1}}/4 = 4/3+3/4<5/2$. The desired result follows immediately.
\end{proof}

\subsection{Proof of Theorem~\ref{thm:main} for the Gaussian Disorder}\label{ssec:pf-gaussian}
By collecting all the lemmas in previous subsections, we give a formal proof of central limit theorem of free energy for the Gaussian disorder using the characteristic approach.

By Theorem~\ref{thm:latala}, we have the uniform boundedness of $N\E\la (R_{12}-q)^2\ra_t$ for all $t\in [0,1]$. Note that, what we need is that for an appropriate choice of $c_t$, we have, for $t>0$
\begin{align}\label{eq:dct}
	\eta_N^{-1}\E \bigl|\la N(R_{12}-q)^2\ra_{t} - c_t\bigr| \to 0
\end{align}
as $N\to\infty$.

By the discussions in Section~\ref{ssec:g-clt}, we need to prove the $L^1$ convergence in~\eqref{eq:l1-conver} for the \subc~and \crit~cases. The inequality~\eqref{eq:cs-ineq} suggests that it is enough to prove the results in~\eqref{eq:subcrit-res}. Those results were established rigorously in the Sections~\ref{ssec:tail-logp} and~\ref{sec:new-pf}, respectively. Therefore, the CLT in Theorem~\ref{thm:main} with Gaussian disorder immediately follows with the asymptotic variance
$$
	v^2 = \frac12{\gb^2\rho^4}/{(1-\gb^2)} -\frac12{\gb^2} -\frac12\log(1-\gb^2).
$$ 

The central limit theorem in the \supc~regime easily follows by taking $g(y):=\exp(\text{\i} xy)$ for $x\in\dR$ and $\eta_N=Nh^4\to\infty$. The scaled variance is given by $v_2^2 = \gb^2/(2(1-\gb^2)).$

\subsection{Proof of Lemmas~\ref{lem:1} and~\ref{lem:2}}\label{pf:auxlem}

In this subsection, we include the proof details for Lemma~\ref{lem:1} and~\ref{lem:2}. To do that, we first need the following elementary two-sided bound on $\tanh(x)$.

\begin{lem}
	We have
	\begin{align*}
		0\le x^2-\tanh^2x\le \frac23x^4 \text{ for all } x\in\dR.
	\end{align*}
\end{lem}
\begin{proof}
	We used the fact that $1-x^2/3\le x^{-1} \tanh x\le 1$ for all $x$. Thus $x^{-2}\tanh^2x\le 1$ or $x^2-\tanh^2x\ge 0$. Similarly,
	\begin{align*}
		x^2-\tanh^2x=x^2(1-x^{-1}\tanh x)(1+x^{-1}\tanh x)\le x^2\cdot x^2/3 \cdot 2 = \frac23x^4.
	\end{align*}
	This completes the proof.
\end{proof}

Now we prove Lemma~\ref{lem:1}.
\begin{proof}[Proof of Lemma~\ref{lem:1}]
	Define $Y=\gb\sqrt{q}g+h$. Note that $\E Y^2=\gb^2q+h^2$. Thus
	$h^2-(1-\gb^2)q = \E(Y^2 -\tanh^2Y) \ge 0$ or $(1-\gb^2)q \le h^2$. Moreover,
	\begin{align*}
		h^2-(1-\gb^2)q = \E(Y^2 -\tanh^2Y)
		 & \le \frac23\E Y^4.
	\end{align*}
	A direct computation yields that
	\begin{align*}
		\E Y^4 = h^4+6\gb^2h^2q +3\gb^4q^2
		 & \le \frac{h^4}{(1-\gb^2)^2}((1-\gb^2)^2+6\gb^2(1-\gb^2) +3\gb^4) \\
		 & =\frac{h^4}{(1-\gb^2)^2}(3-2(1-\gb^2)^2)
		\le \frac{3h^4}{(1-\gb^2)^2},
	\end{align*}
	completing the proof.
\end{proof}

\begin{proof}[Proof of Lemma~\ref{lem:2}]
	By, Poincar\'e inequality for Gaussian distribution we have
	\begin{align*}
		\var(\cosh(\gb\sqrt{q}g+h)) \le \gb^2q\E \tanh^2(\gb\sqrt{q}g+h) = \gb^2q^2.
	\end{align*}
	For any smooth function $f$ and $g,g'\sim \N(0,1)$ i.i.d., with $g_{\theta}=g'\cos\theta + g\sin\theta$ we have
	\begin{align*}
		\var(f(g)) & = \E f(g)(f(g_{1})-f(g_{0}))
		= \int_{0}^{\pi/2} \E\left(f'(g)f'(g_{\theta})\right) \cos\theta d\theta.
	\end{align*}
	For $f(x)=\log\cosh(\gb\sqrt{q}x+h)$, we get
	\begin{align*}
		\var(\cosh(\gb\sqrt{q}g+h)) & = \gb^2q \int_{0}^{\pi/2} \E\left( \tanh(\gb\sqrt{q}g+h) \tanh(\gb\sqrt{q}g_{\theta}+h)\right) \cos\theta d\theta \\
		 & = \gb^2q \int_{0}^{\pi/2} \E\left( (\gb\sqrt{q}g+h)(\gb\sqrt{q}g_{\theta}+h) \right) \cos\theta d\theta + O( \gb^2q^3) \\
		 & = \gb^2q \int_{0}^{\pi/2} \left( h^2+\gb^2q\sin\theta \right) \cos\theta d\theta + O( \gb^2q^3) \\
		 & =\gb^2q (h^2+ \frac12\gb^2q ) + O( \gb^2q^3).
	\end{align*}
	Simplifying we have the result.
\end{proof}   

\section{Cluster Based Approach for General Disorder}\label{sec:cluster}

From the analysis presented in Section~\ref{sec:direct-pf} and~\ref{sec:char}, we can see how the log-partition function behaves in all 3 different regimes. However, the deficiency is the dependence of Gaussian disorder and the characteristic approach in Section~\ref{sec:char} can not give the results up to the critical temperature $\gb_c$ in all regimes. This is partially due to the analytical nature of the approach. In order to overcome those issues, we use the cluster based approach which works by decomposing the configuration space as a cluster of simple structures contributing to the partition function. Besides that, the disorders can be relaxed to general symmetric distribution. This approach was first used by Aizenman, Lebowitz, and Ruelle~\cite{ALR87} to deal with the SK model with no external field. Basically, in the zero external field case, the main contribution to the partition function is from the clusters of simple loops of the underlying complete graph. When the external field is present, we found that a new path cluster beyond loop starts contributing to the partition function. In particular, in the \crit~regime, the path cluster's effects are strong enough to compete with the loop clusters. Nonetheless, in the \subc~regime, the loop cluster plays a dominate role, this gives a different explanation on why the log-partition function in the \subc~regime behaves nearly same as zero external field case as shown in Section~\ref{sec:direct-pf} and~\ref{sec:char}. Similar loop counting technique also appears in some other models, see the results in~\cites{BD21,Ban20,ALS21,KL05,Kos06}. We also present some open questions on how to extend the analysis in this article to those models in Section~\ref{sec:open}.

In the following part, we will rigorously establish the decomposition of the partition function and explain how the path cluster appears. Applying the Stein's method to prove multivariate CLT will be given the in the next Section.

Consider the external field $h=\rho N^{-\ga}$ with $\ga \ge 1/4$. Fix $\gb<1$. Recall that the disorder random variable $J_{ij}$'s are i.i.d.~mean zero and variance one. We define
\begin{align*}
	\bar{Z}_{N}(\gb) & := \prod_{i<j} \cosh(\gb J_{ij}/\sqrt{N}) \\
	\text{ and }
	\hat{Z}_{N}(\gb,h) & := \E_{\mvgs} \prod_{i<j} \bigl(1+ \gs_i\gs_j \tanh(\gb J_{ij}/\sqrt{N})\bigr)
\end{align*}
where $\mvgs$ is a vector of $N$ many i.i.d.~$\pm 1$ valued random variables with mean $\hh$.
From the decomposition in~\eqref{def:Zdecomp}, we have
\begin{align}\label{def:fdecomp}
	Z_{N}(\gb,h)=(2\cosh h)^{N}\cdot \bar{Z}_{N}(\gb) \cdot \hat{Z}_{N}(\gb,h).
\end{align}

Let $K_N$ be the complete graph on $N$ vertices with edge set $\cE_N:=\{e=(i,j)\mid 1\le i<j\le n\}$. For a subset $\gC\subseteq \cE_{N}$, we define the random variable
\begin{align}\label{def:wgC}
	\go(\gC):= \prod_{e\in\gC} \tanh(\gb J_{e}/\sqrt{N}).
\end{align}
We will write $\go_{e}$ for $\go(e)$.

Next, we note that
\begin{align*}
	\log \bar{Z}_{N}(\gb )
	= \sum_{i<j} \log \cosh(\gb J_{ij}/\sqrt{N})
	= -\frac12\sum_{e} \log (1-\go_{e}^2)
\end{align*}
and a standard argument shows that
\begin{align}\label{eq:ZNbar}
	\log \bar{Z}_{N}(\gb ) - \frac12\sum_{e} \go_{e}^2 \to \frac18\gb^{4}\E J_{12}^4
\end{align}
in probability as $N\to\infty$.
Moreover, we have
\begin{align}\label{eq:znhat}
	\hat{Z}_{N}(\gb,h)
	 & = \E_{\mvgs} \prod_{i<j} \left( 1+ \gs_{i}\gs_{j} \tanh(\gb J_{ij}/\sqrt{N}) \right)
	= \sum_{\gC \subseteq \cE_N} \hh^{|\partial\gC|}\cdot \go(\gC)
\end{align}
where the sum is over all subsets $\gC$ of $\cE_{N}$ and $\partial \gC =\{i\in[n]\mid \#\{j\mid (i,j)\in \gC\} \text{ is odd}\}$ is the set of odd degree vertices in $\gC$. 

First we will prove that, in~\eqref{eq:znhat} the sum over all subsets $\gC\subseteq \cE_{N}$, can be approximated by a sum over large finite $\gC$. We define, for any integer $m\ge 1$,
\begin{align*}
	\hat{Z}_{N,m}:= \sum_{\gC \subseteq \cE_N, |\gC| \le m} \hh^{|\partial\gC|}\cdot \go(\gC).
\end{align*}

\begin{lem}\label{lem:big-err}
	Assume that $
		\gb_N^2 = N\E \tanh(\gb J/\sqrt{N})^2 \in (0,1) \text{ and }
		\rho_N^4 =N\hh^4 \in[0,\infty).
	$
	For any $m\ge 2(1+\rho_{N}^4)/(1-\gb_N^2)^2$ we have
	\begin{align*}
		\bigl\Vert \hat{Z}_{N} - \hat{Z}_{N,m} \bigr\Vert_2
		 & \le 2m^{1/8}\gb_N^m\exp\left(\sqrt{{(1+\rho_N^4)m}/2}\right).
	\end{align*}
\end{lem}

\begin{rem}
	From the above $L^2$ bound, if $\rho_N = \infty$, \ie~h belongs to the \supc~regime, the right handside error will not enjoy the exponential decay any more. In other words, in the \supc~regime, very large $\gC$ matters, and counting very large loops and paths in $\gC$ becomes intractable. Another issue is that the Stein's method in Section~\ref{sec:stein} also breaks down.
\end{rem} 

The proof of Lemma~\ref{lem:big-err} is given in Section~\ref{pf:big-err}. Next we prove that $\hat{Z}_{N,m}$ can be approximated by products over simple loops and simple paths of length $\le m$. We make the key approximation step rigorous in Lemma~\ref{lem:approx} below. We define
\begin{align*}
	\tilde{Z}_{N,m}:= \prod_{\gc\,:\,|\gc|\le m} (1+\go(\gc)) \prod_{p\,:\, |p|\le m } (1+\hh^2\cdot \go(p)).
\end{align*}

\begin{lem}\label{lem:approx}
	As $N \to \infty$, for any fixed $m\in\dN$, we have
	\begin{align*}
		\bigl\Vert \hat{Z}_{N,m} - \tilde{Z}_{N,m} \bigr\Vert_2\to 0
	\end{align*}
\end{lem}

The proof is given in Section~\ref{pf:approx}.
Next, we introduce the following lemma to give the asymptotic distribution of the product over loops and paths if length bounded by $m$. First, we note that
\begin{align*}
	\log \tilde{Z}_{N,m} = \sum_{\gc\,:\,|\gc|\le m} \log(1+\go(\gc)) +\sum_{p\,:\, |p|\le m } \log(1+\hh^2\cdot \go(p)).
\end{align*}
Using the expansion $\log(1+x)\approx x-x^{2}/2$ for $x$ small, we get the following result.

\begin{lem}\label{lem:logapprox}
	For any fixed $m\ge 3$, we have
	\begin{align*}
		\log \tilde{Z}_{N,m}
		- \sum_{\gc\,:\,|\gc|\le m} \go(\gc)
		- \sum_{p\,:\, |p|\le m } \hh^2\cdot \go(p)
		\to
		-\frac12\sum_{\ell=3}^{m}\frac{\gb^{2\ell}}{2\ell} -\frac14\sum_{\ell'=2}^{m}\rho^4\gb^{2\ell'}
	\end{align*}
	in probability as $N\to\infty$.
\end{lem}
\begin{proof}[Proof of Lemma~\ref{lem:logapprox}]

	The key step is the approximation
	\[
		\log(1+x) = x - \frac12x^{2} + O(|x|^{3}/(1-|x|))\text{ for } |x|<1.
	\]
	For the loop part, it is easy to check that
	\[
		\var\bigg(\sum_{\gc\,:\,|\gc|\le m} \go(\gc)^{2}\bigg)\le mN^{-1}\E J_{12}^{4}/(1-\gb^{2})\to 0.
	\]
	Here, the weights of two loops will be independent if they are disjoint and for two loops sharing at least one edge the number of choices for vertices is $N^{\text{sum of the lengths}-1}$. Also,
	\[
		\sum_{\gc\,:\,|\gc|\le m} \E\go(\gc)^{2} \to \sum_{\ell=3}^{m}\frac{\gb^{2\ell}}{2\ell}
	\]
	and thus,
	\[
		\sum_{\gc\,:\,|\gc|\le m} \go(\gc)^2 \to \sum_{\ell=3}^{m}\frac{\gb^{2\ell}}{2\ell}
	\]
	in probability as $N\to \infty$. Similarly, for the path contribution we have
	\[\var\biggl(\sum_{p\,:\,|p|\le m} \hh^{4}\cdot\go(p)^{2}\biggr)\le mN^{-1}\rho^{8}\E J_{12}^{4}/(1-\gb^{2})\to 0\]
	and thus,
	\[
		\sum_{p\,:\,|p|\le m} \hh^{4}\cdot\go(p)^{2}
		\to \frac12\sum_{\ell'=2}^{m}\rho^{4}\gb^{2\ell'}
	\]
	in probability as $N\to \infty$. Also,
	\[
		\abs{\log(1+\hh^{2}\go(p)) - \hh^{2}\go(p) +\hh^{4}\go(p)^{2}/2}\le \hh^{6}\go(p)^{2}/(1-\hh^{2})
	\]
	implies that
	\[
		\abs{\sum_{p\,:\, |p|\le m } \log(1+\hh^2\cdot \go(p))-\hh^2\cdot \go(p)+ \frac12\hh^4\cdot \go(p)^{2}} \le
		\frac{\hh^2}{1-\hh^{2}}\cdot \sum_{p\,:\, |p|\le m } \hh^{4}\go(p)^{2}\to 0
	\]
	in probability. For the loops the approximation error goes to zero since $\max_{|\gc|\le m} |\go(\gc)|\to 0$ in probability, as shown in~\cite{ALR87}.
	Combining everything we have the result.
\end{proof}

Finally, using Stein's method for multivariate normal approximation, we prove in Theorem~\ref{thm:stein-mvn} in Section~\ref{sec:stein}, that the $(2m-1)$--dimensional random vector
\begin{align*}
	\left(\sum_{\gc\,:\,|\gc|= \ell} \go(\gc),\ell=3,4,\ldots,m;\;
	\sum_{p\,:\, |p|=\ell' } \hh^2\cdot \go(p), \ell'=1,2,\ldots,m;\;
	\sum_e (\go_e^2- \E\go_{12}^2)
	\right)
\end{align*}
converges in distribution, as $N\to\infty$, to jointly independent normal distribution with variance matrix
\begin{align*}
	\gS=\gS_m:=\diag\left(\gb^{2\ell}/2\ell,\ell=3,4,\ldots,m;\;
	\gb^{2\ell'}/2, \ell'=1,2,\ldots,m;\;
	\frac12\gb^4\var(J_{12}^2)
	\right).
\end{align*}
\subsection{Proof of Lemma~\ref{lem:big-err}}\label{pf:big-err}
Note that
\begin{align*}
	 & \hat{Z}_{N}(\gb,h)^2 \\
	 & \quad= \E_{\mvgs, \mvgt} \prod_{i<j} \left( 1+ (\gs_{i}\gs_{j}+\gt_{i}\gt_{j} ) \tanh(\gb J_{ij}/\sqrt{N}) + \gs_{i}\gs_{j}\gt_{i}\gt_{j} \tanh(\gb J_{ij}/\sqrt{N})^2 \right).
\end{align*}
and its expectation is given by
\begin{align*}
	\psi_{N}(\gb_{N},\rho_N) & := \E_{\mvgs, \mvgt} \prod_{i<j} \left( 1+ \gs_{i}\gs_{j}\gt_{i}\gt_{j} \gb^2_N/N \right) = \sum_{\gC \subseteq \cE_N} (\rho_N^4/N)^{|\partial\gC|/2}(\gb_N^2/N)^{|\gC|}.
\end{align*}
Using the Gaussian trick and simplifying we have
\begin{align*}
	\psi_{N}(\gb_{N},\rho_N)
	 & \le \E_{\mvgs, \mvgt}\exp\biggl( \frac{\gb^2_N}{2N}\biggl(\sum_{i} \gs_i\gt_j\biggr)^2\biggr) \\
	 & = \E\left( \E_{\gs,\gt}\exp(\gb_N/\sqrt{N}\cdot \gs\gt \eta)\right)^{N} \\
	 & = \E\left(\cosh(\eta\gb_N/\sqrt{N})(1+\hh^2\cdot \tanh(\eta \gb_N/\sqrt{N})\right)^{N}.
\end{align*}
Now, using $1+x\le e^{x}$ for all $x$, we have
\begin{align*}
	\psi_{N}(\gb_{N},\rho_N)
	 & \le \E\exp\left(\gb^2_N \eta^2 /2+ \rho_N^2\sqrt{N}\cdot \tanh(\eta \gb_N/\sqrt{N} ) \right) \\
	 & \le 2(1-\gb_{N}^2)^{-1/2}\exp\left(\rho_N^4\gb_{N}^2/2(1-\gb_{N}^2)\right).
\end{align*}
In fact, one can easily check that
\begin{align*}
	\lim_{\gb_{N}\to \gb, \rho_N\to\rho}\psi_{N}(\gb_{N},\rho_N)
	 & = (1-\gb^2)^{-1/2}\exp\left(-\gb^4/4-\gb^2/2 +\rho^4\gb^2/2(1-\gb^2)\right).
\end{align*}
We also have
\begin{align*}
	\sum_{|\gC|\ge m} (\rho_N^4/N)^{|\partial\gC|/2}(\gb_N^2/N)^{|\gC|}
	 & \le e^{-2my}\psi_{N}(e^{y}\gb_{N},\rho_N)
	\\
	 & \le 2\gb_N^{2m}t^{-1/2}\exp\left(-m\log(1-t)+\rho_N^4(1-t)/2t\right)
\end{align*}
for any $t = 1-e^{2y}\gb^2_{N}\in (0,1-\gb_{N}^2)$. Taking
\begin{align*}
	t=\sqrt{\frac{1+\rho_{N}^4}{2m}} \le \frac12 (1-\gb_N^2)<\frac12
\end{align*}
and using $-\log(1-t)\le t + t^2/2(1-t) <t+t^2$
we get
\begin{align*}
	\sum_{|\gC|\ge m} (\rho_N^4/N)^{|\partial\gC|/2}(\gb_N^2/N)^{|\gC|}
	 & \le 2\gb_{N}^{2m}\cdot \left( \frac{2e^2m}{1+\rho_N^4}\right)^{1/4}\exp\left(\sqrt{2(1+\rho_N^4)m}\right) \\
	 & \le 4m^{1/4}\gb_{N}^{2m}\exp\left(\sqrt{2(1+\rho_N^4)m}\right).
\end{align*}
Thus we have the result.\qed 

\subsection{Proof of Lemma~\ref{lem:approx}}\label{pf:approx}
We prove the lemma in two steps. First we break the LHS as follows
\begin{align}
	 & \textbf{LHS} \le \norm{\sum_{\gC \subseteq \cE_N, |\gC| < m} \hh^{|\partial\gC|}\cdot \go(\gC)
		- \sum_{n=0}^{k_N} \sum_{n_1+n_2=n} \sum\!'
	\prod_{i=1}^{n_1}\go(\gc_i)\prod_{j=1}^{n_2}\go(p_j)} \label{eq:decop1} \\
	 & + \norm{\sum_{n=0}^{k_N} \sum_{n_1+n_2=n} \sum\!'
	\prod_{i=1}^{n_1}\go(\gc_i)\prod_{j=1}^{n_2}\go(p_j)- \prod_{\gc: |\gc|\le m} (1+\go(\gc)) \prod_{p: |p|\le m } (1+\hh^2\cdot \go(p))}\label{eq:decop2}
\end{align}
where the sum $\sum'$ is over all collections of $n_{1}$ many distinct loops $\gc_1,\gc_{2},\ldots,\gc_{n_1}$ and $n_{2}$ many distinct paths $p_1,p_{2},\ldots,p_{n_2}$ of length $\le m$; and $k_{N}\to\infty$ as $N\to\infty$.

The first term on the RHS can be proved using combinatorial analysis (extending ALR proof), and the second difference part is proved using series expansion. We start with the second term.

Define the two polynomials,
\begin{align*}
	\phi(z) & := \prod_{\gc:\abs{\gc}\le m} (1+z \go(\gc)) \text{ and }
	\psi(z) := \prod_{p: \abs{p} \le m} (1+z \hh^2\go(p)).
\end{align*}
By the remainder formula, for fixed $R>1$, we have
\begin{align}
	 & \abs{\phi(1) - \sum_{n=0}^k \frac{1}{n!}\phi^{(n)}(0)} \le \frac{1}{(R-1)R^k} \sup_{|z|=R} \abs{\phi(z)}, \\
	 & \abs{\psi(1) - \sum_{n=0}^k \frac{1}{n!}\psi^{(n)}(0)} \le \frac{1}{(R-1)R^k} \sup_{|z|=R} \abs{\psi(z)}.
\end{align}

For $\sup_{|z|=R} \abs{\phi(z)}$, we bound it as follows,
\begin{align*}
	\sup_{\abs{z} =R} \abs{\phi(z)}^2 & = \sup_{\abs{z}=R} \prod (1+z\go(\gc)) (1+\bar{z} \go(\gc)) \\
	 & = \sup_{\abs{z} =R} \prod (1+2\Re(z) \go(\gc) + R^2 \go(\gc)^2) \\
	 & \le \exp \bigl(2R |\sum \go(\gc)| + R^2 \sum\go(\gc)^2\bigr).
\end{align*}
Since, we prove in Section~\ref{sec:stein} that $\sum_{\abs{\gc}\le m} \go(\gc)$ converges in distribution, hence tight and similarly, $\sum_{\abs{\gc}\le m} \go(\gc)^2$ is tight, we get that, when $k_{N}\to\infty$, we have
\begin{align}
	\prod_{\gc:\abs{\gc}\le m} (1+ \go(\gc)) - \sum_{n=0}^{k_N} \sum_{\substack{\gc_1,\gc_{2},\ldots,\gc_n \\ \gc_i\neq \gc_j,\,\abs{\gc_i}\le m}}\prod_{i=1}^n \go(\gc_i) \to 0 \text{ in distribution as } N \to \infty.
\end{align}
Similarly one can prove the convergence for the path cluster,
\begin{align}
	\prod_{p:\abs{p}\le m} (1+ \hh^2\go(p)) - \sum_{n=0}^{k_N} \sum_{\substack{p_1,p_{2},\ldots,p_n \\ p_i\neq p_j,\,\abs{p_i}\le m}}\hh^{2n}\prod_{i=1}^n \go(p_i) \to 0 \text{ in distribution as } N \to \infty.
\end{align}
Combining the above arguments together, the convergence of~\eqref{eq:decop2} to zero is immediate.

The errors in~\eqref{eq:decop1} have two sources, one of which is from the large graphs $\abs{\gC}\ge m$. It was shown in Lemma~\ref{lem:big-err} that this error decays exponentially. The other error is from the multigraphs contribution induced by loop and path clusters. In particular, we need prove the following term is small,
\begin{align}\label{eq:multi-err}
	\sum_{\substack{\abs{\gC'}\le m}} C(\gC') \go(\gC'), \ \text{ where} \ \gC' \ \text{ is multi-graph}.
\end{align}
The combinatorial factor
\begin{align*}
	C(\gC'):=\sum_{\substack{\gc_1 \neq \cdots \neq \gc_\ell, \abs{\gc_i}\le m \\ p_1 \neq \cdots \neq p_{\ell'}, \abs{p_j}\le m \\ \cup_i \gc_i \cup_j p_j = \gC' }} \hh^{2\ell'}.
\end{align*}
Recall that
\begin{align*}
	\go(\gC') = \prod_{e \in \gC'} \tanh(\gb J_{e}/\sqrt{N}) \approx N^{-\abs{\gC'}/2},
\end{align*}
where $\abs{\gC'}$ is the number of edges in $\gC'$ counting multiplicities, the extra factor $1/2$ is due to the fact that single bond gives 0 contribution. The argument for controlling the combinatorial factor $C(\gC')$ is nearly same as the proof of~\cite{ALR87}*{Lemma 3.1}. The only difference is now we have path clusters, the number of ways to choose a path of length $\ell$ is proportional to $N^{\ell+1}$, while for a loop of length $\ell$, there are $N^{\ell}$ ways. The extra $N$ factor will be killed by the weak external field $h^4$. Therefore,
\begin{align*}
	\sum_{\substack{\abs{\gC'}\le m}} C(\gC') \go(\gC') \lesssim \frac 1 N.
\end{align*}
This completes the proof.\qed

\section{Multivariate CLT using Stein's Method}\label{sec:stein}

With the approximation of $\hat{Z}_N$ by clusters of loops and paths in Section~\ref{sec:cluster}, we now apply Stein's
method to prove the multivariate central limit theorem for the log-partition function. Before we state the main theorem, let us define the following set of random variables. We define
\begin{align*}
	\go_{ij}:=\tanh(\gb J_{ij}/\sqrt{N})
\end{align*}
which are i.i.d.~with mean $0$, variance
\begin{align*}
	\gz_{N}^{2}:=\E\tanh(\gb J_{12}/\sqrt{N})^{2}\approx \gb^{2}/N
\end{align*}
and sixth moment bounded by
\begin{align*}
	\E\abs{\go_{12}}^{6}\le \gb^{6}\E J_{12}^{6}/N^{3}.
\end{align*}
As we remark before $(4+\eps)$--th moment bound, for some $\eps>0$, suffices here for the CLT to hold.

We define the following collections of random variables arising from the decomposition of $Z_{N}$ in Section~\ref{sec:cluster}.

\begin{itemize}\setlength{\itemsep}{1em}
	\item {\bf Independent Sum:} We define
	 \begin{align*}
		 Q_N:=\sum_{e} (\go_{e}^{2}-\gz_{N}^{2}),
	 \end{align*}
	 which is the dominating term~\eqref{eq:ZNbar} in $\bar{Z}_N$.

	\item {\bf Loop Cluster:} For $\ell\ge 3$, we define
	 \begin{align*}
		 L_N^\ell:=\sum_{|\gc|=\ell}\prod_{e\in\gC}\go(e)=\sum_{|\gc|=\ell}\go(\gc)
	 \end{align*}
	 as the sum of weights for all loops of length $\ell$ in the complete graph $K_N$.
	\item {\bf Path Cluster:} For $\ell\ge 1$, we define
	 \begin{align*}
		 P_N^\ell:=\hh^2\sum_{|p|=\ell}\prod_{e\in p}\go(e)=\hh^2\sum_{|p|=\ell}\go(p)
	 \end{align*}
	 as the sum of weights for all paths of length $\ell$ in the complete graph $K_N$.
\end{itemize}

It is easy to see that $\E Q_{N}=\E L_{N}^{\ell}=\E P_{N}^{\ell'}=0$ for all $\ell,\ell'$. Define
\begin{align*}
	N_{2}:= \binom{N}{2}.
\end{align*}
We have
\begin{align*}
	\var(Q_N) & = N_{2}\var(\go_{12}^{2}) \to \frac{\gb^4}2\var(J_{12}^{2}), \\
	\var(L_{N}^{\ell}) = \frac{(N)_{\ell}}{2\ell}\cdot \gz_{N}^{\ell}\to \frac{\gb^{2\ell}}{2\ell}
	\quad & \text{ and }\quad
	\var(P_{N}^{\ell'}) =\hh^4\cdot \frac12 (N)_{\ell'}\cdot \gz_{N}^{\ell}\to \rho^4\cdot \frac{\gb^{2\ell'}}{2}
\end{align*}
as $N\to\infty$. Moreover, these random variables are uncorrelated with each other by symmetry of $\go_{ij}$.

We will use multivariate Stein's method in Theorem~\ref{thm:stein-mvn} to prove the random vector formed by finitely many of the above random variables indeed has a asymptotic normal limit. Let us now state our theorem.

\begin{thm}\label{thm:stein-mvn}
	Let $\mvW:=(L_N^{\ell}, 3\le \ell\le m; P_N^{\ell'},1\le \ell'\le m; Q_N)$ be the $(2m-1)$--dimensional random vector with mean 0 and covariance matrix $\gS$ defined above. For any three times differentiable function $f$, we have
	\begin{align*}
		\abs{ \E f(\mvW) - \E f(\gS^{1/2}\mvZ)} \le C\left({\abs{f}_2}N^{-1/2} +{\abs{f}_3}N^{-1/4} \right),
	\end{align*}
	where $C$ is some constant depending only on $m,\gb$ and the distribution of $J_{12}$.
\end{thm}

We apply the multivariate Stein's method in Theorem~\ref{thm:rr-mvstein} to prove Theorem~\ref{thm:stein-mvn}. We first construct an exchangeable pair $(W,\tilde{W})$, where $\tilde{W}$ is obtained by selecting an edge $(I,J), I<J$ uniformly at random from $\cE_N$, and then replacing $\go_{IJ}$ by an i.i.d.~copy $\go'$. Denote $\tilde{W}$ be $W$ w.r.t.~this new disorder. We will use $\tilde{Q}_{N},\tilde{L}_{N}^{\ell},\tilde{P}_{N}^{\ell'}$ de denote the corresponding random variables w.r.t.~the new disorder.

We also define,
\begin{align*}
	\gD Q_{N}:= \tilde{Q}_{N} -Q_{N},\
	\gD {L}_{N}^{\ell}:= \tilde{L}_{N}^{\ell} - {L}_{N}^{\ell},\
	\gD {P}_{N}^{\ell'} := \tilde{P}_{N}^{\ell'} - {P}_{N}^{\ell'}.
\end{align*}
Next we verify the assumptions in Theorem~\ref{thm:rr-mvstein}. Starting at the loop cluster, we define
\begin{align*}
	\sP^{\ell}(i,j) \text{ as the collection of all paths starting at $i$ and ending at $j$ with total length } \ell.
\end{align*}

\subsection{Loop Cluster Analysis}\label{ssec:lca}
Fix $\ell\ge 3$. For the loop clusters, we check the linearity condition first:
\begin{align*}
	\E(\gD L_N^{\ell}\mid \mvW)
	 & = \E((\go'-\go_{IJ}) \sum_{p \in\sP^{\ell-1}(I,J)} \go(p)\mid \mvW) \\
	 & = -\frac{1}{N_2} \sum_{i<j} \go_{ij} \sum_{p \in\sP^{\ell-1}(i,j)} \go(p)	= -\frac{\ell}{ N_2 } L_N^{\ell}.
\end{align*}
Comparing with the linearity condition in~\eqref{eq:stein-linear}, it indicates that
\begin{align*}
	\gl_\gc^{\ell} = {\ell}/{ N_2 }.
\end{align*}

Notice that $\E (\gD L_N^\ell)^2 = 2\gl_{\gc}^{\ell}\var(L_{N}^{\ell}).$
Now we bound the variance of the conditional second moments of the change. The proof is given in Section~\ref{pf:loop2bd}.

\begin{lem}\label{lem:loop2bd}
	We have
	\begin{align*}
		\frac{1}{\gl_\gc^{\ell}} \sqrt{\var \E((\gD L_N^\ell)^2 \mid \mvW)} \lesssim N^{-1/2}.
	\end{align*}
\end{lem}

Now we bound the third moment of the absolute change $|\gD L_{N}^{\ell}|$.
\begin{lem}\label{lem:loop3bd}
	We have $$
		\frac{1}{\gl_\gc^{\ell}}\E \abs{\gD L_N^{\ell}}^3 \lesssim N^{-1/4}.$$
\end{lem}
\begin{proof}
	Recall that $\gD L_N^{\ell} = (\go'- \go_{ij})\sum_{p \in \sP_\gc^{\ell-1}(i,j)}\go(p)$. Thus, using $|\go|\le 1$, we get
	\begin{align*}
		\E \abs{\gD L_N^{\ell}}^3
		 & \le 8 \E \bigl( \abs{\go_{ij}}^3 \cdot \bigl|\sum_{p\in \sP_\gc^{\ell-1}(i,j)}\go(p)\bigr|^3\bigr) \\
		 & \le 8 (\E \go_{12}^6)^{1/4}\cdot \bigl( \E\bigl|\sum_{p\in \sP_\gc^{\ell-1}(i,j)}\go(p)\bigr|^4\bigr)^{3/4} \\
		 & \sim N^{-3/4} (N^{2\ell-4} \cdot N^{-2\ell+2})^{3/4} = \cO(N^{-9/4})
	\end{align*}
	where the second step is by H\"older's inequality. The quantitative control of the sum over $\sP_\gc^{\ell-1}(i,j)$ is based on similar ideas in Lemma~\ref{lem:l2norm}. Note that, here we upper bounded $\go^{12}$ by $\go^{6}$ and used the sixth moment control. If we have finite $(4+\eps)$--th moment, using $(4+\eps)$--th moment control, we will get a bound $N^{-(4+\eps)/8-3/2}=N^{-2-\eps/8}$, thus giving $N^{-\eps/8}$ instead of $N^{-1/4}$ in the Lemma statement.
\end{proof}  

\subsection{Path Cluster Analysis}\label{ssec:pca}

The analysis of the path cluster is quite similar to the loop clusters. For clarity, we will just state the corresponding results about variance and third absolute moment bound below. Proof details are given in Section~\ref{sec:steincomp}. Recall that, the path objects have the form $P_N^{\ell'} := \hh^2 \sum_{|p|=\ell'} \go(p)$ with $\ell'\ge 1$. We denote the difference between the exchangeable pairs as
$
	\gD P_N^{\ell'} = \tilde{P}_N^{\ell'} - P_N^{\ell'}.
$
We use $\sP_p^{\ell'-1}(i,j)$ to denote the collection of paths (may not be connected) of length $\ell'-1$, where each of them is formed by removing the edge $(i,j)$ from a connected path of length $\ell'$. The linearity condition can be verified similarly in this case,
\begin{align*}
	\E(\gD P_N^{\ell'}\mid \mvW)
	 & = \hh^{2}\E\bigl((\go'-\go_{IJ})\sum_{p' \in \sP_p^{\ell'-1}(I,J)} \go(p') \mid \mvW \bigr) \\
	 & = -\frac{\hh^2}{N_2}\sum_{i<j}\E\bigl(\go_{ij}\sum_{p' \in \sP_p^{\ell'-1}(i,j)} \go(p') \mid \mvW \bigr)
	= -\frac{\ell'}{ N_2 } P_N^{\ell'}.
\end{align*}
The $\ell'$ is due to different choices of edges from the path of length $\ell'$, as in the loop case, we have $\gl_p^{\ell'} = \frac{\ell'}{ N_2 }$ in the linearity condition~\eqref{eq:stein-linear}.

For the conditional second moment $\E((\gD P_N^{\ell'})^2 \mid \mvW)$, the following lemma gives the variance bound.
\begin{lem}\label{lem:path2bd}
	We have
	\begin{align*}
		\frac{1}{\gl_p^{\ell'}} \sqrt{ \var\E\left( (\gD P_N^{\ell'})^2\mid \mvW\right) } \lesssim N^{-1/2}.
	\end{align*}
\end{lem}

The proof of this part is similar to the proof of Lemma~\ref{lem:l2norm} for the loop cluster. The details of the combinatorial analysis are bit different, but it turns out that this difference is killed by the weak external field. The next lemma gives the bound for the absolute third moment for the path clusters.
\begin{lem}\label{lem:path3bd}
	We have
	\begin{align*}
		\frac{1}{\gl_p^{\ell'}} \E \abs{\gD P_N^{\ell'}}^3 \lesssim N^{-1/4}.
	\end{align*}
\end{lem}  

\subsection{Independent Sum analysis}\label{ssec:isa}
For the case of independent sum,
\begin{align*}
	Q_N= \sum_{i<j} (\go_{ij}^{2}-\gz_{N}^{2})
\end{align*}
is fairly easy to analyze using Stein's method. For completeness, we quickly check conditions without providing too much details. We temporarily denote $x_{ij} := \go_{ij}^{2}-\gz_{N}^{2}$. Recall that, $\gD Q_N = x' - x_{IJ} $. First the linearity condition is clear as
\begin{align*}
	\E(\gD Q_N \mid \mvW) = -\frac{1}{ N_2 }\sum_{i<j} x_{ij} = -\frac{1}{N_2}Q_N.
\end{align*}
Here, $\gl = \frac{1}{N_2}$. The variance bound is given in the following lemma.
\begin{lem}
	We have
	\begin{align}\label{eq:ind-var}
		\frac1\gl \sqrt{\var \E((\gD Q_N)^2\mid \mvW)}
		\lesssim N^{-1}
		\text{ and }
		\frac1\gl \E \abs{\gD Q_N}^3
		\lesssim N^{-1}.
	\end{align}
\end{lem}
\begin{proof}
	Note that,
	\begin{align*}
		\E((\gD Q_N)^2\mid \mvW) = \E x_{12}^2 + \frac{1}{ N_2 } \sum_{i<j} x_{ij}^2.
	\end{align*}
	Taking the variance, we get
	\begin{align*}
		\var \E((\gD Q_N)^2\mid \mvW) = \frac{1}{N_2} \var(x_{12}^2)\approx N^{-2-4}=N^{-6}.
	\end{align*}
	Similarly, we bound the third moment,
	\begin{align*}
		\E \abs{\gD Q_N}^3 = \frac{1}{ N_2 } \sum_{i<j} \E \abs{x' -x_{ij}}^3 \le \frac{8}{ N_2 } \sum_{i<j} \E \abs{x_{ij}}^3 \lesssim N^{-3}.
	\end{align*}
	The result stated in~\eqref{eq:ind-var} now follows easily.
\end{proof}

\subsection{Mixed Products}
In this subsection, we analyze the mixed product terms between loops, paths, and the independent sum. For clarity, we will state the results first and delay most of the proof details in Section~\ref{sec:steincomp}.

The first lemma is about the covariance bound.
\begin{lem}\label{lem:mix2bd}
	For any $\ell\ge 3,\ell'\ge 1$, we have
	\begin{align*}
		\frac 1 \gl \var \E(\gD P_N^{\ell'} \gD L_N^{\ell}\mid \mvW) \lesssim N^{-1}, \quad & \quad
		\frac 1 \gl \var \E(\gD P_N^{\ell'} \gD Q_N\mid \mvW) \lesssim N^{-1}, \\
		\text{ and } \frac 1 \gl \var \E(\gD L_N^{\ell'} \gD Q_N\mid \mvW) & \lesssim N^{-1}.
	\end{align*}
\end{lem}

The following lemma gives the mixed third moment control.
\begin{lem}\label{lem:mix3bd}
	For $k_1,k_2,k_3 \in\{0,1,2\}$ such that $k_1+k_2+k_3=3$, we have
	\begin{align*}
		\frac1\gl\E \bigl|(\gD L_N^{\ell})^{k_1}(\gD P_N^{\ell'})^{k_2}(\gD Q_N)^{k_3}\bigr| \le CN^{-1/4},
	\end{align*}
	where $\ell\ge 3$ and $\ell'\ge 1$ are fixed.
\end{lem} 

Collecting all the above estimates, we now apply Theorem~\ref{thm:rr-mvstein} to prove the Theorem~\ref{thm:stein-mvn}.  

\begin{proof}[Proof of Theorem~\ref{thm:stein-mvn}]
	By all the variance and covariance bound, we can get an estimate of the constant $A$, where all $\gl$'s are of order $N^{2}$. The dominate variance terms are due to the variance of loop and path clusters, which are of order $N^{-1/2}$, thus
	\begin{align*}
		A & \le \sum_{\ell=3}^{m}\sum_{\ell'=1}^m N_2 \sqrt{\var \E(\gD P_N^{\ell'} L_N^{\ell}\mid \mvW)} \\
		 & \qquad + \sum_{\ell'=1}^m N_2 \left( \sqrt{\var \E((\gD P_N^{\ell'})^2\mid \mvW)} + \sqrt{\var \E(\gD P_N^{\ell'} \gD Q_N\mid \mvW)} \right) \\
		 & \quad \quad + \sum_{\ell=3}^{m} N_2 \left(\sqrt{\var \E((\gD L_N^{\ell})^2\mid \mvW)} + \sqrt{\var \E(\gD L_N^{\ell} \gD Q_N\mid \mvW)} \right) \\
		 & \lesssim N^{-1/2}.
	\end{align*}

	Similarly, by collecting all the third moment estimates, the constant $B$ can be bounded as follows,
	\begin{align*}
		B & \le (1+m)^3\cdot \max(\E|(\gD L_N^{\ell})|^{3}, \E|(\gD P_N^{\ell'})|^{3},\E|(\gD Q_N)|^{3})
		\le C_{m,\gb} N^{-1/4}
	\end{align*}
	where $C_{m,\gb}$ is some constant depending on $m$ and $\gb$. By applying Theorem~\ref{thm:rr-mvstein}, we have for any thrice differentiable function $f$,
	\begin{align*}
		\abs{\E f(\mvW) - \E f(\gS^{1/2}\mvZ)} \le C_{m,\gb}\frac{\abs{f}_2}{\sqrt{N}} + C_{m,\gb} \frac{\abs{f}_3}{N^{1/4}}.
	\end{align*}
	This completes the proof.
\end{proof}

\section{Computations for Stein's Method}\label{sec:steincomp}

This section is devoted to the proof details from Section~\ref{sec:stein}, where multivariate Stein's method is applied to establish the normal approximation for the free energy.

\subsection{Proofs for Loop Clusters}

\subsubsection{Proof of Lemma~\ref{lem:loop2bd}}\label{pf:loop2bd}
First we observe that
\begin{align*}
	\E((\gD L_N^\ell)^2 \mid \mvW) - \E(\gD L_N^\ell)^2
	 & = \E\biggl( (\go' - \go_{IJ})^2\cdot \biggl(\sum'_{p} \go(p)\biggr)^2 \;\biggl|\; \mvW\biggr) - \frac{(N)_{\ell}\gz_{N}^{\ell}}{N_2} \\
	 & = \frac{1}{ N_2 } \E\biggl(\sum_{i<j} (\gz_N^2+\go_{ij}^2 )\cdot \biggl(\sum'_{p} \go(p)\biggr)^2 \;\biggl|\; \mvW\biggr) - \frac{(N)_{\ell}\gz_{N}^{\ell}}{N_2} \\
	 & = \frac{1}{ N_2 }\E(A_\gc^{\ell} + 2B_\gc^{\ell} +C_\gc^{\ell} +2D_\gc^{\ell} \mid \mvW),
\end{align*}
where the sum $\gS^{'}$ is over all paths in $\sP^{\ell-1}(i,j)$ and
\begin{alignat*}{3}
	A_{\gc}^{\ell} & := \sum_{i<j}(\go_{ij}^2-\gz_N^2) \sum'_{p} \go(p)^2,
	\qquad & B_{\gc}^{\ell} & :=\gz_N^2 \sum_{i<j}\sum'_{p}(\go(p)^2 - \gz_N^{2(\ell-1)}), \\
	C_{\gc}^{\ell} & := \sum_{i<j}(\go_{ij}^2-\gz_N^2) \sum'_{p_1\neq p_2} \go(p_1)\go(p_2),
	\text{ and }
	 & D_{\gc}^{\ell} & := \gz_N^2 \sum_{i<j}\sum'_{p_1\neq p_2} \go(p_1)\go(p_2).
\end{alignat*}
Here the first equality is by applying the conditional expectation, averaging over the choice of $(I,J)$, using the independence property of $\go_{ij}$ and $\go'$. The second equality is just to break the second moments part into several different cases. Now we prove that all the terms above are in a smaller scale.

\begin{lem}\label{lem:l2norm}
	We have
	\begin{align*}
		\norm{A_{\gc}^{\ell}}_2^2 + \norm{B_\gc^{\ell}}_2^2 + \norm{C_{\gc}^{\ell}}_2^{2} & \lesssim N^{-2}, \text{ and }
		\norm{D_{\gc}^{\ell}}_2^{2} \lesssim N^{-1}.
	\end{align*}
\end{lem}

Using Lemma~\ref{lem:l2norm}, we get
\begin{align*}
	\sqrt{\var\E((\gD L_N^\ell)^2 \mid \mvW)}
	 & = \norm{ \E((\gD L_N^\ell)^2 \mid \mvW) - \E(\gD L_N^\ell)^2 }_2^2 \\
	 & \lesssim N^{-2}(\norm{A_{\gc}^{\ell}}_2 + \norm{B_\gc^{\ell}}_2 + \norm{C_{\gc}^{\ell}}_2+\norm{C_{\gc}^{\ell}}_2) \lesssim N^{-5/2}.
\end{align*}
Recall that, $\gl_\gc^{\ell} = \frac{\ell}{N_2} = O(N^{-2})$, the announced results follows.
Now, we provide the proof of Lemma~\ref{lem:l2norm}.
\begin{proof}[Proof of Lemma~\ref{lem:l2norm}]
	Let $\xi_{ij}:=\go_{ij}^2 - \gz_N^2$. It is clear that $\E \xi_{ij} = 0$. For $A_{\gc}^{\ell}$, we have
	\begin{align*}
		\norm{A_{\gc}^{\ell}}_2^2
		 & = \E\sum_{i_1<j_1, i_2<j_2} \xi_{i_1j_1}\xi_{i_2j_2} \sum'_{p_1, p_2} \go(p_1)^2 \go(p_2)^2
	\end{align*}
	where $p_{t}\in \sP^{\ell-1}(i_{t},j_{t}), t=1,2$.
	By the fact that $\E \xi_{ij} = 0$, the contributions to the right hand side sum can be divided into the following two cases.
	\begin{itemize}
		\item Case 1: $e_1 \neq e_2$, but the edge $e_1:=(i_1,j_1)$ must also appear in the path $p_2$, and similarly edge $e_1:=(i_2,j_2)$ must also appear in the path $p_1$.
		\item Case 2: Edge $e_1:=(i_1,j_1)$ is the same as edge $e_2:=(i_2,j_2)$.
	\end{itemize}
	For Case 1, the contribution is
	\begin{align*}
		\E \sum_{\substack{i_1<j_1, \\i_2<j_2}} \xi_{i_1j_1}\go_{i_1,j_1}^2\xi_{i_2j_2}\go_{i_2,j_2}^2 \sum_{\substack{ \{p_1' \cup e_2\} \in\sP_\gc^{\ell-1}(i_1,j_1),\\ \{ p_2' \cup e_1\} \in\sP_\gc^{\ell-1}(i_2,j_2)}} \go(p_1' )^2 \go(p_2')^2.
	\end{align*}

	There is a further subtle difference, $e_1$ and $e_2$ can share at most 1 vertex under the assumption $e_1\neq e_2$. If $\abs{\{i_1,j_1\} \cap \{i_2,j_2 \}}=1$, then the total number of loops is of the order $N^3 \cdot N^{2(\ell-3)}$, and each loop has contribution $N^{-2\ell}$, thus the total contribution for this subcase is $N^{-3}$. While for the other subcase $\abs{\{i_1,j_1\} \cap \{i_2,j_2 \}}=0$, \ie~the 4 vertices $i_1,i_2,j_1,j_2$ are all different. The number of ways to form the loops is of the order $N^4 \cdot N^{2(\ell-4)}$, each loop has contribution $N^{-2\ell}$, the total contribution is $N^{-4}$. We implicitly used the property that $\gz_{N}^{2}=\E \go_{ij}^2 \approx N^{-1}$.

	Now we analyze Case 2, where $e_1=e_2$. The contribution is given by
	\begin{align*}
		\E \sum_{i<j} \xi_{ij}^2 \sum_{p_1,p_{2} \in \sP_\gc^{\ell-1}(i,j)} \go(p_1)^2\go(p_2)^2
		\approx N^2 \cdot N^{2(\ell-2)} \cdot N^{-2\ell} = N^{-2}.
	\end{align*}

	In a similar fashion, we prove the $L^2$ bound for $C_\gc^{\ell}, D_\gc^{\ell}$. First for $C_\gc^{\ell}$,
	\begin{align*}
		\norm{C_{\gc}^{\ell}}_2^{2}
		=
		\E \sum_{\substack{i_1<j_1, \\ i_2 < j_2}} \xi_{i_1j_1} \xi_{i_2,j_2} \sum_{\substack{p_1 \neq p_2 \in\sP_\gc^{\ell-1}(i_1,j_1),\\ p_3 \neq p_4 \in\sP_\gc^{\ell-1}(i_2,j_2)}} \prod_{s=1}^4\go(p_s).
	\end{align*}
	The right hand side can be divided into 2 cases as before. For the case where $\abs{\{i_1,j_1\} \cap \{i_2,j_2 \}}= 0$, due to the fact that $\E \xi_{ij} = 0$, the edge $e_1$ must appear in the paths $p_3,p_4$, and edge $e_2$ must appear in the paths $p_1,p_2$. For the fixed $N^4$ different choices of $i_1,j_1,i_2,j_2$, one has total $N^{4(\ell-4)}$ different paths in the second sum, but the edges in those paths must appear at least twice. This leads to the total combinatorial factor is bounded by $N^4 \cdot N^{2\ell -8}$, then the total contribution is of the order $N^{-4}$. For the subcase $\abs{\{i_1,j_1\} \cap \{i_2,j_2 \}}= 1$, the total contribution is of the order $N^3 \cdot N^{2\ell - 6} \cdot N^{-2\ell} = N^{-3}$.

	When $\abs{\{i_1,j_1\} \cap \{i_2,j_2 \}}= 2$, \ie~ $e_1=e_2$, using similar analysis based on observations that each edge except $(i,j)$ appear at least twice in the second sum and $p_1 \neq p_2, p_3 \neq p_4$.
	\begin{align*}
		\E \sum_{i<j} \xi_{ij}^2 \sum_{\substack{p_1 \neq p_2 \in\sP_\gc^{\ell-1}(i,j), \\ p_3 \neq p_4 \in\sP_\gc^{\ell-1}(i,j)}} \prod_{s=1}^4\go(p_s)
		\lesssim N^2 \cdot N^{2\ell -4} \cdot N^{-2\ell} = N^{-2}.
	\end{align*} 

	For $D_\gc^{\ell}$, we have
	\begin{align*}
		\norm{D_{\gc}^{\ell}}_2^{2}
		= \gz_N^4 \E \sum_{\substack{i_1<j_1, \\i_2<j_2}}\sum_{\substack{p_1 \neq p_2 \in \sP_\gc^{\ell-1}(i_1,j_1),\\
				p_3 \neq p_4 \in \sP_\gc^{\ell-1}(i_2,j_2)}} \prod_{s=1}^4 \go(p_s).
	\end{align*}

	Comparing to the bound for $C_{\gc}^{\ell}$, the difference is that the assumption $e_1 \in p_3,p_4$ and $e_2 \in p_1,p_2$ are no longer needed. For $\abs{\{i_1,j_1\} \cap \{i_2,j_2 \}}\ge 1$, using similar ideas as before, one can still get an upper bound of the total contribution as $N^3 \cdot N^{2\ell-5} \cdot N^{-2\ell} = N^{-2} $. However for $\abs{\{i_1,j_1\} \cap \{i_2,j_2 \}} =0$, the analysis becomes more delicate. Since $p_1 \neq p_2 \in \sP_\gc^{\ell-1}(i_1,j_1)$, and $p_3 \neq p_4 \in \sP_\gc^{\ell-1}(i_2,j_2)$, there exist at least one vertex $i_*$ where the edge $e_*:=(i_*,j_*) \in p_1$ is different from the edge $e_{\Box}:=(i_*,j_{\Box})\in p_2$. Thus the edges $e_*, e_{\Box}$ must also appear in the path $p_3,p_4$, which decrease the number of freedom to choose vertices in path $p_3,p_4$ by 1. Therefore, the total number of ways in the sum will be at most $N^4 \cdot N^{\ell-2} \cdot N^{\ell-3}$, then the overall contribution will be at most $N^{-1}$ in this case.

	Finally, we deal with $B_\gc^{\ell}$,
	\begin{align*}
		\norm{B_{\gc}^{\ell}}_2^{2} & = \gz_N^4 \E \sum_{\substack{i_1<j_1, \\i_2<j_2}} \sum_{\substack{p_1 \in \sP_\gc^{\ell-1}(i_1,j_1),\\p_2 \in \sP_\gc^{\ell-1}(i_2,j_2)}} ( \go(p_1)^2 - \gz_N^{2(\ell-1)})( \go(p_2)^2 - \gz_N^{2(\ell-1)}).
	\end{align*}
	For the summand, we have
	\begin{align*}
		\gz_N^4\E \left(( \go(p_1)^2 - \gz_N^{2(\ell-1)})( \go(p_2)^2 - \gz_N^{2(\ell-1)}) \right)
		 & = \gz_N^4(\E \go(p_1)^2\go(p_2)^2 - \gz_N^{4(\ell-1)}) \\
		 & \le C_{\gb, \alpha} N^{-2-2(\ell-1)} \lesssim N^{-2}\cdot N^{-2(\ell-1)},
	\end{align*}
	using the independence property and sixth moment assumption. Number of pairs of paths with non-empty intersection in the sum for $\norm{B_{\gc}^{\ell}}_2^{2}$ is bounded by $N^{2\ell-2}$. This completes the proof.
\end{proof}   

\subsection{Proofs for Path Clusters}
We first prove the variance bound for the path clusters in Lemma~\ref{lem:path2bd}.

\subsubsection{Proof of Lemma~\ref{lem:path2bd}}\label{pf:path2bd}
Notice that
\begin{align*}
	\E((\gD P_N^{\ell'})^2) & = \hh^4\E (\go'-\go_{ij})^2 \bigl(\sum_{p' \in \sP_{p}^{\ell'-1}(i,j)}\go(p') \bigr)^2
	= \frac{2\hh^4}{ N_2 } \sum_{i<j}\gz_N^2 \sum_{p' \in \sP_{p}^{\ell'-1}(i,j)} \gz_N^{2(\ell'-1)}.
\end{align*}
Then the centered conditional second moment can be split into several terms as follows,
\begin{align*}
	 & \E((\gD P_N^{\ell'})^2\mid \mvW) - \E((\gD P_N^{\ell'})^2) \\
	 & = \E\biggl( (\go_{ij}'-\go_{ij})^2 \biggl( \sum_{p' \in \sP_p^{\ell'-1}(i,j)} \go(p') \biggr)^2 \,\biggl|\, \mvW\biggr) - \E((\gD P_N^{\ell'})^2) \\
	 & = \frac{\hh^4}{ N_2 } \sum_{i<j} \E\biggl( (\gz_N^2+\go_{ij}^2) \biggl( \sum_{p' \in \sP_p^{\ell'-1}(i,j)} \go(p') \biggr)^2 \,\biggl|\, \mvW\biggr)- \E((\gD P_N^{\ell'})^2) \\
	 & := \frac{1}{ N_2 }\E(A_p^{\ell'} + 2B_p^{\ell'} +C_p^{\ell'} +2D_p^{\ell'} \mid \mvW),
\end{align*}
where
\begin{alignat*}{1}
	A_p^{\ell'} & := \sum_{i<j}\hh^4(\go_{ij}^2-\gz_N^2) \sum_{p' \in \sP_p^{\ell'-1}(i,j)} \go(p')^2, \\
	B_p^{\ell'} & := \gz_N^2 \hh^4\sum_{i<j}\sum_{p' \in \sP_p^{\ell'-1}(i,j)} (\go(p')^2 -\gz_N^{2(\ell'-1)}), \\
	C_p^{\ell'} & :=\sum_{i<j}\hh^4(\go_{ij}^2-\gz_N^2) \sum_{p_1'\neq p_2' \in \sP_p^{\ell'-1}(i,j)} \go(p_1')\go(p_2'), \\
	\text{ and } D_p^{\ell'} & := \gz_N^2 \hh^4 \sum_{i<j}\sum_{p_1'\neq p_2' \in \sP_p^{\ell'-1}(i,j)} \go(p_1')\go(p_2').
\end{alignat*} 

\begin{lem}\label{lem:l2-path}
	We have
	\begin{align*}
		\norm{A_p^{\ell'} }_2^2 +\norm{B_p^{\ell'} }_2^2+ \norm{C_p^{\ell'} }_2^2 \lesssim N^{-2}\text{ and } \norm{D_p^{\ell'} }_2^2 \lesssim N^{-1}.
	\end{align*}
\end{lem}

Now we use Lemma~\ref{lem:l2-path} to prove the variance bound in Lemma~\ref{lem:path2bd}.

Using the lemma of second moment bounds, one can get
\begin{align*}
	\var\left( \E((\gD P_N^{\ell'})^2 \mid \mvW)\right) & = \norm{ \E((\gD P_N^{\ell'})^2 \mid \mvW) - \E (\gD P_N^{\ell'})^2}_2^2 \lesssim N^{-5}
\end{align*}
Recall that $\gl_p = \frac{\ell'}{ N_2 } = O(N^{-2})$, it is clear that $$\frac{1}{\gl_p} \sqrt{ \var\left( \E((\gD P_N^{\ell'})^2\mid \mvW)\right) } \lesssim N^{-1/2}.$$ Thus, we are left with proving Lemma~\ref{lem:l2-path}.

\begin{proof}[Proof of Lemma~\ref{lem:l2-path}]
	Recall that $\xi_{ij}:= \go_{ij}^2 -\gz_N^2$. For $A_p^{\ell'}$, we have
	\begin{align*}
		\norm{A_p^{\ell'}}_2^2
		= \hh^8 \E\sum_{\substack{i_1<j_1, \\i_2<j_2.}} \xi_{i_1,j_1} \xi_{i_2,j_2} \sum_{\substack{p_1' \in \sP_p^{\ell'-1}(i_1,j_1),\\ p_2' \in \sP_p^{\ell'-1}(i_2,j_2)}} \go(p_1')^2 \go(p_2')^2.
	\end{align*}

	The analysis is similar as in Lemma~\ref{lem:l2norm}, but now there are in total $N^{\ell'+1}$ different ways to form a path of length $\ell'$. Basically for all the sub-cases analyzed in Lemma~\ref{lem:l2norm}, the combinatorial factor now will have higher power by 2.

	\begin{itemize}
		\item For the case $\abs{\{i_1,j_1\} \cap \{i_2,j_2 \}}=0$, the contribution is $\hh^8 \cdot N^4 \cdot N^{2(\ell'-3)} \cdot N^{-2\ell'} \approx N^{-4} $.
		\item For the case $\abs{\{i_1,j_1\} \cap \{i_2,j_2 \}}=1$, the contribution is $\hh^8 \cdot N^3 \cdot N^{2(\ell'-2)} \cdot N^{-2\ell'} \approx N^{-3} $.
		\item For the case $\abs{\{i_1,j_1\} \cap \{i_2,j_2 \}}=2$, the contribution is $\hh^8 \cdot N^2 \cdot N^{2(\ell'-1)} \cdot N^{-2\ell'} \approx N^{-2} $.
	\end{itemize}

	Similarly for $C_p^{\ell'}$, we have
	\begin{align*}
		\norm{C_p^{\ell'}}_2^2 = \E \hh^8\sum_{\substack{i_1<j_1, \\i_2<j_2.}} \xi_{i_1,j_1}\xi_{i_2,j_2} \sum_{\substack{p_1'\neq p_2' \in \sP_p^{\ell'-1}(i_1,j_1), \\p_3'\neq p_4' \in \sP_p^{\ell'-1}(i_2,j_2) }} \prod_{s=1}^4 \go(p_i').
	\end{align*}
	\begin{itemize}
		\item For the case $\abs{\{i_1,j_1\} \cap \{i_2,j_2 \}}=0$, the contribution is upper bounded by $\hh^8 \cdot N^4 \cdot N^{2(\ell'-3)} \cdot N^{-2\ell'} \approx N^{-4} $.
		\item For the case $\abs{\{i_1,j_1\} \cap \{i_2,j_2 \}}=1$, the contribution is upper bounded by $\hh^8 \cdot N^3 \cdot N^{2(\ell'-2)} \cdot N^{-2\ell'} \approx N^{-3} $.
		\item For the case $\abs{\{i_1,j_1\} \cap \{i_2,j_2 \}}=2$, the contribution is upper bounded by $\hh^8 \cdot N^2 \cdot N^{2(\ell'-1)} \cdot N^{-2\ell'} \approx N^{-2} $.
	\end{itemize}

	The term $D_p^{\ell'}$ is more delicate in a similar fashion of $D_\gc$.
	\begin{align*}
		\norm{D_p^{\ell'}}_2^2 = \E \gz_N^4 \hh^8 \sum_{\substack{i_1<j_1, \\i_2<j_2.}} \sum_{\substack{p_1'\neq p_2' \in \sP_p^{\ell'-1}(i_1,j_1), \\p_3'\neq p_4' \in \sP_p^{\ell'-1}(i_2,j_2) }} \prod_{s=1}^4 \go(p_i').
	\end{align*}
	For the case $\abs{\{i_1,j_1\} \cap \{i_2,j_2 \}}=0$, we have to use the fact $p_1' \neq p_2'$, $p_3' \neq p_4'$. Again, due to the triangular structures, there must be at least 1 vertex, where the path $p_1'$ and $p_2'$ branch. To get nonzero contribution, that vertex must also be used in paths $p_3',p_4'$. This decrease the number of freedom to choose the paths $p_3',p_4'$. The other cases $\abs{\{i_1,j_1\} \cap \{i_2,j_2 \}}\ge 1$ can be easily controlled as before. Thus
	\begin{itemize}
		\item For the case $\abs{\{i_1,j_1\} \cap \{i_2,j_2 \}}=0$, the contribution is upper bounded by $\hh^8 \cdot N^4 \cdot N^{2(\ell'-3)} \cdot N^{-2\ell'} \approx N^{-1} $.
		\item For the case $\abs{\{i_1,j_1\} \cap \{i_2,j_2 \}}=1$, the contribution is upper bounded by $\hh^8 \cdot N^3 \cdot N^{2(\ell'-2)} \cdot N^{-2\ell'} \approx N^{-2} $.
		\item For the case $\abs{\{i_1,j_1\} \cap \{i_2,j_2 \}}=2$, the contribution is upper bounded by $\hh^8 \cdot N^2 \cdot N^{2(\ell'-1)} \cdot N^{-2\ell'} \approx N^{-2} $.
	\end{itemize}
	For $B_p^{\ell'}$, the proof is almost same as in Lemma~\ref{lem:l2norm} by using the independence property and fourth moment assumption.
\end{proof}

Now we prove the absolute third moment in Lemma~\ref{lem:path3bd}.
\subsubsection{Proof of Lemma~\ref{lem:path3bd}}\label{pf:path3bd}
Recall that, $\gD P_N^{\ell} = \hh^2(\go_{ij}'- \go_{ij})\sum_{p' \in \sP_p^{\ell'-1}(i,j)}\go(p')$. By H\"older's inequality,
\begin{align*}
	\E \abs{\gD P_N^{\ell}}^3
	 & \le 8\hh^{6} (\E \go_{e}^6)^{1/4} \biggl(\E \bigl|\sum_{p'\in \sP_p^{\ell'-1}(i,j)}\go(p')\bigr|^4 \biggr)^{3/4} \\
	 & \approx N^{-3/2-3/4} \cdot (N^{2\ell-2} \cdot N^{-2\ell+2})^{3/4} = N^{-9/4}.
\end{align*}
The quantitative control of the sum over $\sP_p^{\ell'-1}(i,j)$ is based on similar ideas in Lemma~\ref{lem:l2-path}. 

\subsection{Proofs for Mixed Products}

In this part, we include the proof details for Lemma~\ref{lem:mix2bd} and Lemma~\ref{lem:mix3bd}.

\subsubsection{Proof of Lemma~\ref{lem:mix2bd}}\label{pf:mix2bd}

First, for the co-variance bound between loops and paths, we start splitting the conditional second moment as 

\begin{align*}
	\E (\gD L_N^{\ell} P_N^{\ell'}\mid \mvW)
	 & = \E \biggl( (\go' -\go_{ij})^2 \hh^2 \sum_{p\in \sP_\gc^{\ell-1}(i,j)}\sum_{p' \in \sP_p^{\ell'-1}(i,j)} \go(p) \go(p') \,\biggl|\, \mvW\biggr) \\
	 & = \frac{1}{ N_2 } \E \biggl( \sum_{i<j} \hh^2(\gz_N^2 + \go_{ij}^2)\sum_{p\in \sP_\gc^{\ell-1}(i,j)}\sum_{p' \in \sP_p^{\ell'-1}(i,j)} \go(p) \go(p') \,\biggl|\, \mvW\biggr) \\
	 & := \frac{1}{ N_2 }\E(A_{\gc p}^{\ell\ell'} + 2B_{\gc p}^{\ell\ell'} \mid \mvW)
\end{align*}
where
\begin{align*}
	A_{\gc p}^{\ell\ell'} & = \sum_{i<j} \hh^2( \go_{ij}^2 -\gz_N^2) \sum_{p\in \sP_\gc^{\ell-1}(i,j)}\sum_{p' \in \sP_p^{\ell'-1}(i,j)} \go(p) \go(p') \\
	B_{\gc p}^{\ell\ell'} & =\sum_{i<j}\gz_N^2\hh^2 \sum_{p\in \sP_\gc^{\ell-1}(i,j)}\sum_{p' \in \sP_p^{\ell'-1}(i,j)} \go(p) \go(p').
\end{align*}

Since in this case, the structure of $p$ and $p'$ must be different, we split the conditional covariance into two terms. Next we prove that the $L^2$ norm is small.

It is clear that $\E(\gD P_N^{\ell'} \gD L_N^{\ell}) = 0$, since paths and loops have different structure with at least one free edge.

\begin{lem}\label{lem:l2-cross}
	We have
	\begin{align*}
		\norm{A_{\gc p}^{\ell\ell'}}_2^2 \le N^{-2}\text{ and }
		\norm{B_{\gc p}^{\ell\ell'}}_2^2 & \le N^{-1}
	\end{align*}
\end{lem}
\begin{proof}[Proof of Lemma~\ref{lem:l2-cross}]
	For $A_{\gc p}^{\ell\ell'}$,
	\begin{align}
		\norm{A_{\gc p}^{\ell\ell'}}_2^2 = \E \sum_{\substack{i_1<j_1, \\i_2<j_2}} \hh^4 \xi_{i_1,j_1} \xi_{i_2,j_2} \sum_{\substack{p_1 \in \sP_{\gc}^{\ell-1}(i_1,j_1),\\ p_2 \in \sP_{\gc}^{\ell-1}(i_2,j_2)}}\sum_{\substack{p_1' \in \sP_p^{\ell'-1}(i_1,j_1),\\ p_2' \in \sP_p^{\ell'-1}(i_2,j_2)}}\go(p_1)\go(p_2)\go(p_1')\go(p_2').
	\end{align}
	Since paths and loops have different structures, one can only pair path (loop) and path (loop) to form a nonzero contributions. Thus for $\abs{\{i_1,j_1\} \cap \{i_2,j_2 \}} \le 1$, $p_1\setminus e_2$ must share same edges as $p_2 \setminus e_1$, and similarly for $p_1',p_2'$. The total contribution in this case is bounded by $N^3 \cdot N^{-1} \cdot N^{-4} \cdot N^{\ell-3} N^{\ell-2} N^{-2(\ell-2)} = N^{-3}$. For $\abs{\{i_1,j_1\} \cap \{i_2,j_2 \}} =2$, the contribution is $N^2 \cdot N^{-3} \cdot N^{(\ell-2)+(\ell-1)} \cdot N^{-2(\ell-1)}= N^{-2}$.

	For $B_{\gc p}^{\ell\ell'}$,
	\begin{align}
		\norm{B_{\gc p}}_2^2 = \E 4 \gz_N^4 \hh^4\sum_{\substack{i_1<j_1, \\ i_2< j_2}}\sum_{\substack{p_1 \in \sP_{\gc}^{\ell-1}(i_1,j_1),\\ p_2 \in \sP_{\gc}^{\ell-1}(i_2,j_2)}}\sum_{\substack{p_1' \in \sP_p^{\ell'-1}(i_1,j_1),\\ p_2' \in \sP_p^{\ell'-1}(i_2,j_2)}}\go(p_1)\go(p_2)\go(p_1')\go(p_2').
	\end{align}

	As we did in the proof of Lemma~\ref{lem:l2norm} for $D_{\gc}^{\ell}$, when $\abs{\{i_1,j_1\} \cap \{i_2,j_2 \}} \ge 1$, the contribution is bounded by $N^{-3} \cdot N^{-2(\ell-1)} \cdot N^3 \cdot N^{(\ell-2)+(\ell-1)} = N^{-1}$. However, for the case $\abs{\{i_1,j_1\} \cap \{i_2,j_2 \}} =0$, where the number of ways to choose $i_1,j_1,i_2,j_2$ is around $N^4$, which leads to a constant bound for $B_{\gc p}$. Therefore, we have to use the fact that $p_1 \neq p_1'$, this implies that there exist at least 1 vertex where $p_1$ and $p_1'$ branches, \ie not sharing edges starting at that point. To obtain nonzero contribution, that vertex must also appear in $p_2,p_2'$. This decreases the total number of freedom at least by 1. Therefore, the overall contribution is still upper bounded by $N^{-1}$.
\end{proof}

Similarly, we expand the term $\E(\gD L_N^\ell \gD Q_N)$ and give the proof of $L^2$ bound. The case for path cluster $P_N^{\ell'}$ with $Q_N$ can be carried out with minor modification.

\begin{align*}
	\E(\gD L_N^\ell \gD Q_N\mid \mvW)
	 & = \E \biggl( (\go' -\go_{IJ}) (x'-x_{IJ}) \sum_{p \in \sP_\gc^{\ell-1}(I,J)} \go(p) \,\biggl|\, \mvW\biggr) \\
	 & = \frac{1}{ N_2 } \E \biggl( \sum_{i<j} \go_{ij}x_{ij} \sum_{p \in \sP_\gc^{\ell-1}(i,j)} \go(p)\,\biggl|\, \mvW\biggr) .
\end{align*}
It is easy to check that $\E (\gD L_N^{\ell} \gD Q_N) = 0$.
Then we derive the covariance bound.
\begin{lem}\label{lem:loop-ind-cov}
	We have
	\begin{align}
		\frac 1 \gl \sqrt{\var \E(\gD L_N^{\ell} \gD Q_N\mid \mvW)} \lesssim N^{-1}.
	\end{align}
\end{lem}

\begin{proof}
	We have
	\begin{align*}
		\var \E(\gD L_N^{\ell} \gD Q_N\mid \mvW)
		 & = \frac{1}{ N_2^2} \E \sum_{\substack{i_1<j_1, \\i_2<j_2}} \go_{i_1,j_1} \go_{i_2,j_2} x_{i_1,j_1} x_{i_2,j_2}\sum_{\substack{p_1\in \sP_\gc^{\ell-1}(i_1,j_1), \\ p_2\in \sP_\gc^{\ell-1}(i_2,j_2)}} \go(p_1) \go(p_2) \\
		 & = \frac{1}{ N_2 ^2} \E \sum_{i<j} \go_{ij}^2 x_{ij}^2 \sum_{ p\in \sP_\gc^{\ell-1}(i,j)} \go(p)^2 \approx N^{-6}.
	\end{align*}
	In the third equality, we used the fact that with $\E \go_{ij}x_{ij} = 0$, the case of different edges $(i_1,j_1)$ and $(i_2,j_2)$ will give zero contribution. We also implicitly use the fact $\E(\go_{ij}^2) \approx N^{-1}$ and $\E(x_{ij}^2) \approx N^{-2}$, to complete the proof.
\end{proof}

\begin{rem}
	The analysis of the co-variance $\var \E(\gD P_N^{\ell'} \gD Q_N\mid \mvW)$ is nearly same as the Lemma~\ref{lem:loop-ind-cov}. The only difference is that in the path case, the number of ways to choose a path $p' \in \sP_p^{\ell'-1}(i,j)$ such that $(i,j) \cup p'$ is a path of length $\ell'$ is $N^{\ell'-1}$. In loop case, we have already seen it's $N^{\ell-2}$. However, due to the effect of weak external field $h$, the order of co-variance will still be $N^{-6}$.
\end{rem} 

\subsubsection{Proof of Lemma~\ref{lem:mix3bd}}\label{pf:mix3bd}

Now we prove the mixed third moments bound in Lemma~\ref{lem:mix3bd}.
Using the fact that $k_1+k_2+k_3 =3$ and H\"older's inequality we get
\begin{align*}
	\E \bigl|(\gD L_N^{\ell})^{k_1}(\gD P_N^{\ell'})^{k_2}(\gD Q_N)^{k_3} \bigr|
	\le \max(\E|(\gD L_N^{\ell})|^{3}, \E|(\gD P_N^{\ell'})|^{3},\E|(\gD Q_N)|^{3}).
\end{align*}
Combining the individual bounds we get the result.

\section{Extension of Cluster Expansion to Other Models}\label{sec:extension}

This section discusses how the cluster-based approach can be easily extended to some other spin glass models to obtain similar fluctuation results under weak external fields. We mainly focus on the bipartite SK model and the diluted SK model.

\subsection{Bipartite SK model}\label{sec:bsk}
The bipartite SK model has received much interest in the past ten years, mainly due to some peculiar features displayed in this model, such that most of the classical proofs do not work anymore. For example, rigorously computing the limiting free energy at low temperature is still in mystery, and the classical sub-additive argument for proving the infinite volume limit is broken. For various aspects of this model, we refer readers to~\cites{DW20,Mou20,BSS19,BGG11,BGGPT14} and references therein. We further point out that the bipartite model can be realized as an instance of the MSK model, see~\cite{DW20}.

Let $N=N_1+N_2$ with $N_{1}/N\to p_{1}, N_{2}/N\to p_{2}$ as $N\to \infty$ for some fixed numbers $p_{1},p_{2}>0$ with $p_{1}+p_{2}=1$. For $(\mvgs, \mvgt) \in \{-1,+1\}^{N_1} \times \{-1,+1\}^{N_2}$, the Hamiltonian of the bipartite model is given by
\begin{align}
	H_{N_1,N_2}(\mvgs,\mvgt) := \frac{\gb}{\sqrt{N}} \sum_{i=1}^{N_1} \sum_{j=1}^{N_2} J_{ij} \gs_i \gt_j + h \sum_{i=1}^{N_1} \gs_i +h \sum_{j=1}^{N_2} \gt_i,
\end{align}
where $J_{ij}$ represent disorders. As in the previous section, we can assume that it has a symmetric distribution with $\E J_{ij}=0, \E J_{ij}^2 =1$ and finite $(4+\eps)$--th moment for some $\eps>0$. The fluctuation results were obtained in~\cite{DW20} for $h>0$ and in~\cite{Liu21} for $h=0$. For $h=h_{N}$, our main goal is to illustrate how the cluster based approach can be easily extended to the current setting, which gives the fluctuation results under weak external fields.

The partition function now has the following form
\begin{align*}
	Z_{N_1,N_2} := \sum_{(\mvgs,\mvgt)} \exp\left(\frac{\gb}{\sqrt{N}} \sum_{i=1}^{N_1} \sum_{j=1}^{N_2} J_{ij} \gs_i \gt_j + h \sum_{i=1}^{N_1} \gs_i +h \sum_{j=1}^{N_2} \gt_i\right).
\end{align*}
Similar decomposition of $Z_{N_1,N_2}$ can be derived,
\begin{align*}
	Z_{N_1,N_2}= (2\cosh h)^{N}\cdot \bar{Z}_{N_1,N_2} \cdot \hat{Z}_{N_1,N_2},
\end{align*}
where
\begin{align*}
	\bar{Z}_{N_1,N_2} = \prod_{e \in \cE_{N_1,N_2}} \cosh\left({\gb J_e}/{\sqrt{N}}\right)\text{ and } \hat{Z}_{N_1,N_2}:=\sum_{\gC \subseteq \cE_{N_1,N_2}}\hh^{\abs{\partial \gC}} \cdot \go(\gC).
\end{align*}
Compared to the SK model, the main difference lies in $\hat{Z}_{N_1,N_2}$, where counting cycles and paths in the complete bipartite graph $\cE_{N_1,N_2}$ is different but still tractable.

The edges will be present in an alternating way. The loop lengths are all even and for the path cluster the endpoints could belong to different parts. This counting needs a bit more careful analysis.

Similar to the analysis in Section~\ref{sec:stein} for SK model,
there are three terms contributing to the fluctuation for the log-partition function. Define $\go_{e}:=\tanh(\gb J_{ij}/\sqrt{N})$ for an edge $e=(i,j)\in \cE_{N_1,N_2}$. We define the following.
\begin{itemize}\setlength{\itemsep}{1em}
	\item {\bf Independent Sum:} We define
	 \begin{align*}
		 Q_N:=\sum_{i=1}^{N_1}\sum_{j=1}^{N_2} (\go_{ij}^{2}-\gz_{N}^{2}).
	 \end{align*}
	 where $\gz_{N}^{2} = \E \go_{11}^{2}$.

	\item {\bf Loop Cluster:} For $\ell\ge 2$, we define
	 \begin{align*}
		 L_N^\ell:=\sum_{|\gc|=2\ell} \go(\gc)
	 \end{align*}
	 as the sum of weights for all loops of length $2\ell$ in the complete bipartite graph $K_{N_1,N_2}$. Note that, here all loops are of even size.

	\item {\bf Path Cluster:} For $\ell\ge 1$, we define
	 \begin{align*}
		 P_N^\ell:=\hh^2\sum_{|p|=\ell} \go(p)
	 \end{align*}
	 as the sum of weights for all paths of length $\ell$ in the complete graph $K_N$.
\end{itemize}

It is easy to see that $\E Q_{N}=\E L_{N}^{\ell}=\E P_{N}^{\ell'}=0$ for all $\ell,\ell'$.
We have
\begin{align*}
	\var(Q_N) & = N_{1}N_2\var(\go_{12}^{2}) \to {\gb^4p_1p_2}\var(J_{12}^{2}), \\
	\var(L_{N}^{\ell}) & = \frac{(N_1)_{\ell}(N_2)_{\ell}}{2\ell}\cdot \gz_{N}^{2\ell}\to \frac{\gb^{4\ell} p_1^\ell p_2^\ell}{2\ell} \\
	\var(P_{N}^{2\ell'}) & = \hh^4\cdot \frac{(N_1)_{\ell'}(N_2)_{\ell'+1} +(N_1)_{\ell'+1}(N_2)_{\ell'}}{2}\cdot \gz_{N}^{2\ell} \to \rho^4 \cdot \frac{\gb^{4\ell'} p_1^{\ell'}p_2^{\ell'} }{2} \\
	\text{and }
	\var(P_{N}^{2\ell'-1}) & = \hh^4\cdot (N_1)_{\ell'}(N_2)_{\ell'}\cdot \gz_{N}^{2\ell'-1}
	\to \rho^4\cdot \gb^{2(2\ell'-1)} p_1^{\ell'}p_2^{\ell'}
\end{align*}
for $\ell,\ell'\ge 1$, as $N\to\infty$. Moreover, these random variables are uncorrelated with each other by symmetry of $\go_{ij}$. Summing, we can see that
\begin{align*}
	\text{all the loops contribute } & -\frac12\log(1-\gb^4p_1p_2) - \frac12\gb^4p_1p_2, \\
	\text{all the even length paths contribute } & +\frac12\rho^4\cdot \frac{\gb^4p_1p_2}{1-\gb^4p_1p_2}, \\
	\text{and all the odd length paths contribute } & +\rho^4\cdot \frac{\gb^2p_1p_2}{1-\gb^4p_1p_2}
\end{align*}
to the variance.
With a similar decomposition of the log-partition as in~\eqref{eq:cluster-decomp} obtained (as done to Section~\ref{sec:cluster}), the rest is the same as the proof in Section~\ref{sec:stein}.

\subsection{Diluted SK Model}\label{sec:dsk}

Instead of enabling all interactions among spins, the diluted version of the SK model is now defined on a random graph. The spins interact with a random number of spins. Specifically, we introduce the following SK model defined on a Bernoulli random graph. For a configuration $\mvgs \in \{ -1,+1\}^N$,

\begin{align*}
	H_N(\mvgs)= \sum_{i<j} \eps_{ij}^{(N)} J_{ij} \gs_i \gs_j + h \sum_{i} \gs_i,
\end{align*}
where $\eps_{ij}^{(N)} \sim \text{ Bernoulli}(p/N)$ are i.i.d.~for some $p>0$. It is clear that the average degree of the underlying graph is around $p$, which is fundamentally different from the SK model where the degree is $N$. The disorder $J_{ij}$ has i.i.d.~symmetric distribution as before, and all random variables are independent. In~\cite{Kos06}, H\"olger applied the cluster expansion approach to this model with zero external field and obtained fluctuation results.

Similar decomposition of the partition function $Z_N(\gb,h)$ with $h=0$ gives:
\begin{align*}
	\bar{Z}_{N}(\gb,0) & = \prod_{i<j} \cosh(\gb \eps_{ij}^{(N)}J_{ij})
\end{align*}
and
\begin{align*}
	\hat{Z}_{N}(\gb,0) & = \E_{\mvgs} \prod_{i<j}(1+\gs_i\gs_j \tanh(\gb \eps_{ij}^{(N)} J_{ij})) = \sum_{\substack{\gC \subseteq \cE_N, \abs{\partial \gC}=0}} \prod_{e \in \gC} \tanh(\gb J_e)\cdot \eps_{e}^{(N)}
\end{align*}

The main theorem in~\cite{Kos06} is as follows.
\begin{thm}[\cite{Kos06}*{Theorem 1.1}]\label{thm:dilu-noh}
	For each $k\ge 3$, let $\tilde{Q}_k$ be the distribution of the random variable
	\begin{align*}
		\log \left(1+ \prod_{i=1}^k\tanh(\gb J_{i}) \right),
	\end{align*}
	where $J_i$'s are i.i.d.~random variables with same distribution as the disorder. Let $L_k, k\ge 3$ be independent compound Poisson distributed random variables with Poisson parameter $p^k/2k$ and compounding distribution $\tilde{Q}_k$, respectively. Further assume that
	\begin{align}\label{eq:dilu-temp}
		\hat{p}:=p \E \tanh^2(\gb J) <1.
	\end{align}
	Then
	\begin{align*}
		\log \hat{Z}_{N}(\gb,0)-N\log2 \implies \sum_{k\ge 3}L_{k} \text{ in distribution } \text{ as } N \to \infty.
	\end{align*}
\end{thm}

\begin{rem}
	The limiting distribution of $\log Z_N(\gb,0)$ in the high temperature regime~\eqref{eq:dilu-temp} is still Gaussian. Since, in the diluted SK model, the fluctuation order of $\hat{Z}_N(\gb,0)$ are now much smaller than $\bar{Z}_N(\gb,0)$. Thus the above compound Poisson distribution can not be seen in the limiting distribution of $ \log Z_N(\gb,0)$. This is in sharp contrast to the SK model. For details see~\cite{Kos06}*{Theorem 1.2}.
\end{rem}

Based on the analysis in previous sections, we know that the external field contributes to the partition function mainly through the object $\hat{Z}_N$. Therefore, we only need to figure out what are the analogous results of Theorem~\ref{thm:dilu-noh} after adding a weak external field. It turns out that our analysis can be easily adapted to this case. Here is the main theorem.

\begin{thm}\label{thm:dilu-h}
	Suppose~\eqref{eq:dilu-temp} holds, and $hN^{1/4}\to \rho \in[0,\infty)$. Then
	\begin{align*}
		\log (Z_N(\gb,h)) - N \log(2\cosh h) + v^{2}/2 \implies \sum_{k\ge 3}L_k + v\eta
	\end{align*}
	in distribution as $N \to \infty$,
	where $\eta\sim\N(0,1)$ is independent of $L_{k}$'s and $$v^2 := \frac{\rho^{4}\hat{p}}{2(1-\hat{p})}.$$
\end{thm}

The independent sum of compound Poisson and Gaussian comes up for the same reason as before, \ie~the sum of path and loop contributions to $\hat{Z}_N(\gb,h)$. The average degree of loops is $1$, whereas the paths of length $k$ has average degree $k/(k+1)$. Thus the scaling of the loop cluster for a Poisson limit is too strong for the path cluster, which causes an averaging effect on the path contribution. In particular, expected number of paths of length $k$ in Erd\"os-R\'enyi random graph $G(N,p/N)$ is $\frac12(N)_{k+1}\cdot (p/N)^k \approx Np^k/2$, whereas for loops of length $k$ the expected number is $(N)_{k}/2k\cdot (p/N)^k \approx p^k/2k$. With the critical scaling from $\hh^2\approx \rho^2/\sqrt{N}$, this explains why the path contribution is Gaussian. The mean and variance for this part can be computed similarly as in SK model and bipartite SK model,
\begin{align*}
	\sum_{k\ge 1}\sum_{\pi\,:\,\abs{\pi}=k} \log (1+\hh^2 \cdot \go(\pi))\cdot \prod_{e\in \pi} \eps_e^{(N)}
	\implies &
	\N\left( -\frac12\sum_{k\ge 1} \frac{\rho^4\hat{p}^{k}}{2},\sum_{k\ge 1} \frac{\rho^4\hat{p}^{k}}{2}\right) \\
	 &
	=\N\left(-\frac12 v^2, v^2\right).
\end{align*}

To make the proof rigorous, the first step is to replace $\eps_e^{(N)}$ by $(\eps_e^{(N)}-p/N)+p/N$ and expand the product 
\[
\prod_{e\in \pi}\eps_e^{(N)} = \sum_{S\subseteq \pi}(p/N)^{|\pi|-|S|}\prod_{e\in S} (\eps_e^{(N)}-p/N).
\]
A simple variance computation shows that the non-zero contribution comes only from  $\prod_{e\in \pi} (\eps_e^{(N)}-p/N)$, product of i.i.d.~random variables with mean $0$ and variance $\approx p/N$. Now Stein's method can be applied to the random vector
\[
\left( \sum_{\pi\,:\,\abs{\pi}=k} \hh^2 \cdot \go(\pi)\cdot \prod_{e\in \pi} (\eps_e^{(N)}-p/N)\right)_{1\le k\le m}.
\]
Note that, the random variables $\tilde\go_{ij}:=\tanh(\gb J_{ij})(\eps_{ij}^{(N)}-p/N)$ are i.i.d.~with mean $0$, variance $p\E(\tanh^2(\gb J))/N=\hat{p}/N$ and this explains the restriction~\eqref{eq:dilu-temp} on $\hat{p}$. We leave the proof details to the interested readers. 

As far as we know, this is the first fluctuation result for diluted spin glass models in the presence of an external field. Most of the existing results~\cites{PT04, GT04, FL03} concerning diluted models are either restricted to the $h=0$ case or the first-order asymptotic problem. Theorem~\ref{thm:dilu-h} gives a more refined picture on how the external field affect the fluctuation of $\hat{Z}_N(\gb,h)$ when $hN^{1/4} \in [0,\infty)$, but the overall limit is still Gaussian due to the dominance of $\bar{Z}_N(\gb,h)$. It will be very interesting to see if the external field is very strong such that $\hat{Z}_N(\gb,h)$ can really compete with $\bar{Z}_N(\gb,h)$, then what is the limit distribution of $\log Z_N(\gb,h)$.

\section{Further Questions}\label{sec:open}
In this part, we present several open questions for future research.

\subsection{Cluster Expansion in the \supc~regime}
The essential part of this work is generalizing the cluster expansion approach to the spin glass models with a weak external field, where we found a new combinatorial structure contributing to the partition function, \ie~the path cluster. This gives a new perspective on how the external field affects the spin glass system. However, this approach can only work in the \subc~and \crit~ regimes; The central open question will be extending the current machinery to \supc~case. From Lemma~\ref{lem:big-err}, we can see that the difficult part in \supc~regime is the contribution from large graphs. Counting paths and loops in an infinite graph is intractable. Besides that, Stein's method will not work for infinite many paths and loops.   

\subsection{Fluctuations near critical temperature with external field}
The results in this article are mainly focused on the high temperature regime, \ie~ $\gb<\gb_c$. A natural question would be understanding what happens at $\gb = \gb_c$, but this is a significantly challenging question due to the transition to $\gb>\gb_c$ where the structure is very complicated. The correct fluctuation order is yet to be proved, and we invite interested readers to see the most recent progress in~\cite{CL19}. Instead of focusing exactly on $\gb_c$, one can let $\gb=\gb_N<\gb_c$ depending on $N$ and consider the regime where $\gb_N \to \gb_c$ as $N\to \infty$. In~\cite{Tal11b}*{Theorem 11.7.1}, when $h=0$, it is known that the high temperature behavior continues as long as $\gb_c- \gb_N \gg N^{-1/3}$. It will be interesting to investigate the behavior of the model near $\gb_c$ when the external field is present as the weak form in Assumption~\ref{ass:h}.

Besides that, another collection of interesting questions is applying the path and loop analysis to the models that behave quite differently from the classical SK model, such as Ising Perceptron, SK model with ferromagnetic interactions, etc. For those models, we believe that the cluster-based approach could still be adapted to study the weak external field problems, but possibly in a substantially different manner. The following text briefly introduces those models and presents some heuristics.

\subsection{Ising Perceptron Model}
Ising perceptron model~\cite{GD88} is a toy model of the single-layer neural network. There are several different definitions of this model in the literature. We sketch the following version, the so-called symmetric binary perceptron (SBP) model. Let $J$ be an $M \times N$ matrix with being i.i.d.~Bernoulli random variables taking $\pm 1$. For a fixed $\kappa>0$, the random half-spaces are defined by
\begin{align*}
	\textstyle S_{j}(J):= \left\{ \mvgs \in \gS_N: N^{-1/2} \abs{\sum_{i=1}^N J_{j,i}\gs_i} \le \kappa \right\}, \text{ for } j=1,2,\ldots, M.
\end{align*}
and the capacity by
\begin{align*}
	\textstyle \max \left\{M\ge 1: \bigcap_{j=1}^M S_j(J) \neq \emptyset \right\}.
\end{align*}
From a statistical physics perspective, in order to understand the capacity, it is important to first comprehend the partition function,
\begin{align*}
	\textstyle Z_{M,N}(J):= \abs{\bigcap_{j=1}^M S_j(J)} = 2^N\pr_0\left(\bigcap_{j=1}^M S_j(J)\right),
\end{align*}
where $\pr_0$ is the uniform measure on the $N$-dimensional hypercube. In~\cite{ALS21}, the authors used similar cycle counting techniques in~\cites{ALR87,Ban20} to derive a log-normal limit theorem for $Z_{M,N}(G)$. This limit theorem can be used to study many other properties of the model, which helps to confirm several conjectures.

We consider the following variant of the SBP model by adding an external field. Let $h$ be the weak external field as stated in the Assumption~\ref{ass:h}. Consider the following partition function
\begin{align*}
	\textstyle Z_{M,N}(J,h):= (2\cosh h)^N\cdot \pr_{h}\left(\bigcap_{j=1}^M S_j(J)\right),
\end{align*}
the probability measure $\pr_h$ now is tilted by $h$, and $\pr_h(\gs_i)= \tanh(h)$. By adapting the analysis in this article, a log-normal fluctuation of $Z_{M,N}(J,h)$ with different mean and variance under the weak external field is expected.
\subsection{SK Model with Ferromagnetic Interaction}

For $\mvgs \in \{-1,+1 \}^N$, the Hamiltonian of SK model with ferromagnetic interaction (SKFI) is
\begin{align}\label{eq:SKF}
	H_N^{FI}(\mvgs) = \frac{\gb}{\sqrt{N}} \sum_{i<j} J_{ij} \gs_i\gs_j + h \sum_{i=1}^N\gs_i + \frac{K}{N}\sum_{i<j} \gs_i\gs_j,
\end{align}
where $K>0$ is the coupling constant for the ferromagnetic interaction, the rest is the SK Hamiltonian. In the spherical setting with $h=0$, Baik and Lee~\cite{BL17} proved that depending on the values of $\gb, K$, there are three different regimes: spin glass regime, paramagnetic regime, and ferromagnetic regime. They also establish limit theorems in all those regimes. The technique is based on the tools from random matrix theory as in~\cites{BL16, BCLW21, BL18}. The problem becomes more challenging in the Ising spin case since the random matrix tools do not apply anymore. Banerjee~\cite{Ban20} uses a combinatorial approach to prove a CLT for the free energy in the paramagnetic regime. The technique is based on some cycle counting analysis. A collection of random variables named as signed cycles approximates the free energy. This is of the same nature as in~\cite{ALR87}. However, all the above results are for $h=0$; it will be interesting to investigate what happens if some weak external field is present in the SKFI model. Is there some notion of ``path" components contributing to the free energy? How does it affect the phase diagram of the model? 

\subsection{Diluted $p$-spin Model}
Although remarkable progress has been made on the mean-field spin glass theory in the past decades, a complete rigorous treatment of diluted spin glass models is still far less known. Most importantly, computing the low temperature asymptotic is still very challenging. A generalized Parisi Ansatz for the diluted model was proposed in~\cite{MP01}, some progress towards this Ansatz has been made in~\cites{Pan14di, Pan16di}. Besides that, there is nearly nothing known about the fluctuation of free energy in the diluted $p$-spin models at any temperature and any external field.

For $\mvgs \in \{-1,+1\}^N$, the diluted $p$-spin Hamiltonian is defined as
\begin{align*}
	H_N^{DP}(\mvgs)= \sum_{i_1<i_2<\cdots < i_p} \eps_{i_1i_2\cdots i_p}^{(N)} J_{i_1i_2\cdots i_p} \gs_{i_1} \gs_{i_2}\cdots \gs_{i_p},
\end{align*}
where $\eps_{i_1i_2\cdots i_p}^{(N)}$ are i.i.d Bernoulli. Compared to the diluted SK model in Section~\ref{sec:extension}, the difference is that the spin interaction involves $p>2$ spins. Applying the cluster based approach, one needs to count the number of ``cycles" and ``paths" in a random hypergraph, where the hyperedges are formed by $p$ vertices. Due to the scarcity of results, any progress will be interesting along this direction. However, counting random hypergraphs becomes a very different matter. 

\medskip
\noindent\textbf{Acknowledgements.} The authors would like to thank Daesung Kim, Kesav Krishnan and Greg Terlov for many stimulating discussions. 

\bibliography{wsk.bib}
\end{document}